\documentclass[12pt]{amsart}       
\usepackage{txfonts}
\usepackage{amssymb}
\usepackage{eucal}
\usepackage{graphicx}
\usepackage{amsmath}
\usepackage{amscd}
\usepackage[all]{xy}           

\usepackage{amsfonts,latexsym}
\usepackage{xspace}
\usepackage{epsfig}
\usepackage{float}
\usepackage{color}
\usepackage{fancybox}
\usepackage{colordvi}
\usepackage{multicol}
\usepackage{colordvi}
\usepackage[active]{srcltx} 
\usepackage[colorlinks,final,backref=page,hyperindex,hypertex]{hyperref}

\newtheorem{theorem}{Theorem}[section]
\newtheorem{prop}[theorem]{Proposition}
\theoremstyle{definition}
\newtheorem{defn}[theorem]{Definition}
\newtheorem{lemma}[theorem]{Lemma}
\newtheorem{coro}[theorem]{Corollary}
\newtheorem{prop-def}{Proposition-Definition}[section]
\newtheorem{coro-def}{Corollary-Definition}[section]

\newtheorem{remark}[theorem]{Remark}

\newcommand{\nc}{\newcommand}
\nc{\tred}[1]{\textcolor{red}{#1}}
\nc{\tblue}[1]{\textcolor{blue}{#1}}
\nc{\tgreen}[1]{\textcolor{green}{#1}}
\nc{\tpurple}[1]{\textcolor{purple}{#1}}
\nc{\btred}[1]{\textcolor{red}{\bf #1}}
\nc{\btblue}[1]{\textcolor{blue}{\bf #1}}
\nc{\btgreen}[1]{\textcolor{green}{\bf #1}}
\nc{\btpurple}[1]{\textcolor{purple}{\bf #1}}

\newcommand{\efootnote}[1]{}

\renewcommand{\textbf}[1]{}

\newcommand{\delete}[1]{}

\nc{\mlabel}[1]{\label{#1}}  
\nc{\mcite}[1]{\cite{#1}}  
\nc{\mref}[1]{\ref{#1}}  
\nc{\mbibitem}[1]{\bibitem{#1}} 

\delete{
\nc{\mlabel}[1]{\label{#1}  
{\hfill \hspace{1cm}{\small\tt{{\ }\hfill(#1)}}}}
\nc{\mcite}[1]{\cite{#1}{\small{\tt{{\ }(#1)}}}}  
\nc{\mref}[1]{\ref{#1}{{\tt{{\ }(#1)}}}}  
\nc{\mbibitem}[1]{\bibitem[\bf #1]{#1}} 
}

\nc{\opa}{\ast} \nc{\opb}{\odot} \nc{\op}{\bullet} \nc{\pa}{\frakL}
\nc{\arr}{\rightarrow} \nc{\lu}[1]{(#1)} \nc{\mult}{\mrm{mult}}
\nc{\diff}{\mathfrak{Diff}}
\nc{\opc}{\sharp}\nc{\opd}{\natural}
\nc{\ope}{\circ}
\nc{\dpt}{\mathrm{d}}
\nc{\AW}{\mathcal{A}}

\nc{\ari}{\mathrm{ar}}

\nc{\lef}{\mathrm{lef}}

\nc{\Sh}{\mathrm{ST}}

\nc{\Cr}{\mathrm{Cr}}

\nc{\st}{{Schr\"oder tree}\xspace}
\nc{\sts}{{Schr\"oder trees}\xspace}

\nc{\vertset}{\Omega} 

\nc{\assop}{\quad \begin{picture}(5,5)(0,0)
\line(-1,1){10}
\put(-2.2,-2.2){$\bullet$}
\line(0,-1){10}\line(1,1){10}
\end{picture} \quad \smallskip}

\nc{\operator}{\begin{picture}(5,5)(0,0)
\line(0,-1){6}
\put(-2.6,-1.8){$\bullet$}
\line(0,1){9}
\end{picture}}

\nc{\idx}{\begin{picture}(6,6)(-3,-3)
\put(0,0){\line(0,1){6}}
\put(0,0){\line(0,-1){6}}
 \end{picture}}

\nc{\pb}{{\mathrm{pb}}}
\nc{\Lf}{{\mathrm{Lf}}}

\nc{\lft}{{left tree}\xspace}
\nc{\lfts}{{left trees}\xspace}

\nc{\fat}{{fundamental averaging tree}\xspace}

\nc{\fats}{{fundamental averaging trees}\xspace}
\nc{\avt}{\mathrm{Avt}}

\nc{\rass}{{\mathit{RAss}}}

\nc{\aass}{{\mathit{AAss}}}

\nc{\vin}{{\mathrm Vin}}    
\nc{\lin}{{\mathrm Lin}}    
\nc{\inv}{\mathrm{I}n}
\nc{\gensp}{V} 
\nc{\genbas}{\mathcal{V}} 
\nc{\bvp}{V_P}     
\nc{\gop}{{\,\omega\,}}     

\nc{\bin}[2]{ (_{\stackrel{\scs{#1}}{\scs{#2}}})}  
\nc{\binc}[2]{ \left (\!\! \begin{array}{c} \scs{#1}\\
    \scs{#2} \end{array}\!\! \right )}  
\nc{\bincc}[2]{  \left ( {\scs{#1} \atop
    \vspace{-1cm}\scs{#2}} \right )}  
\nc{\bs}{\bar{S}} \nc{\cosum}{\sqsubset} \nc{\la}{\longrightarrow}
\nc{\rar}{\rightarrow} \nc{\dar}{\downarrow} \nc{\dprod}{**}
\nc{\dap}[1]{\downarrow \rlap{$\scriptstyle{#1}$}}
\nc{\md}{\mathrm{dth}} \nc{\uap}[1]{\uparrow
\rlap{$\scriptstyle{#1}$}} \nc{\defeq}{\stackrel{\rm def}{=}}
\nc{\disp}[1]{\displaystyle{#1}} \nc{\dotcup}{\
\displaystyle{\bigcup^\bullet}\ } \nc{\gzeta}{\bar{\zeta}}
\nc{\hcm}{\ \hat{,}\ } \nc{\hts}{\hat{\otimes}}
\nc{\barot}{{\otimes}} \nc{\free}[1]{\bar{#1}}
\nc{\uni}[1]{\tilde{#1}} \nc{\hcirc}{\hat{\circ}} \nc{\lleft}{[}
\nc{\lright}{]} \nc{\lc}{\lfloor} \nc{\rc}{\rfloor}
\nc{\curlyl}{\left \{ \begin{array}{c} {} \\ {} \end{array}
    \right .  \!\!\!\!\!\!\!}
\nc{\curlyr}{ \!\!\!\!\!\!\!
    \left . \begin{array}{c} {} \\ {} \end{array}
    \right \} }
\nc{\longmid}{\left | \begin{array}{c} {} \\ {} \end{array}
    \right . \!\!\!\!\!\!\!}
\nc{\onetree}{\bullet} \nc{\ora}[1]{\stackrel{#1}{\rar}}
\nc{\ola}[1]{\stackrel{#1}{\la}}
\nc{\ot}{\otimes} \nc{\mot}{{{\boxtimes\,}}}
\nc{\otm}{\overline{\boxtimes}} \nc{\sprod}{\bullet}
\nc{\scs}[1]{\scriptstyle{#1}} \nc{\mrm}[1]{{\rm #1}}
\nc{\margin}[1]{\marginpar{\rm #1}}   
\nc{\dirlim}{\displaystyle{\lim_{\longrightarrow}}\,}
\nc{\invlim}{\displaystyle{\lim_{\longleftarrow}}\,}
\nc{\mvp}{\vspace{0.3cm}} \nc{\tk}{^{(k)}} \nc{\tp}{^\prime}
\nc{\ttp}{^{\prime\prime}} \nc{\svp}{\vspace{2cm}}
\nc{\vp}{\vspace{8cm}} \nc{\proofbegin}{\noindent{\bf Proof: }}
\nc{\proofend}{$\blacksquare$ \vspace{0.3cm}}
\nc{\modg}[1]{\!<\!\!{#1}\!\!>}
\nc{\intg}[1]{F_C(#1)} \nc{\lmodg}{\!
<\!\!} \nc{\rmodg}{\!\!>\!}
\nc{\cpi}{\widehat{\Pi}}
\nc{\sha}{{\mbox{\cyr X}}}  
\nc{\shap}{{\mbox{\cyrs X}}} 
\nc{\shpr}{\diamond}    
\nc{\shp}{\ast} \nc{\shplus}{\shpr^+}
\nc{\shprc}{\shpr_c}    
\nc{\msh}{\ast} \nc{\zprod}{m_0} \nc{\oprod}{m_1}
\nc{\vep}{\varepsilon} \nc{\labs}{\mid\!} \nc{\rabs}{\!\mid}

\nc{\mmbox}[1]{\mbox{\ #1\ }} \nc{\fp}{\mrm{FP}}
\nc{\rchar}{\mrm{char}} \nc{\End}{\mrm{End}} \nc{\Fil}{\mrm{Fil}}
\nc{\Mor}{Mor\xspace} \nc{\gmzvs}{gMZV\xspace}
\nc{\gmzv}{gMZV\xspace} \nc{\mzv}{MZV\xspace}
\nc{\mzvs}{MZVs\xspace} \nc{\Hom}{\mrm{Hom}} \nc{\id}{\mrm{id}}
\nc{\im}{\mrm{im}} \nc{\incl}{\mrm{incl}} \nc{\map}{\mrm{Map}}
\nc{\mchar}{\rm char} \nc{\nz}{\rm NZ} \nc{\supp}{\mathrm Supp}

\nc{\Alg}{\mathbf{Alg}} \nc{\Bax}{\mathbf{Bax}} \nc{\bff}{\mathbf f}
\nc{\bfk}{{\bf k}} \nc{\bfone}{{\bf 1}} \nc{\bfx}{\mathbf x}
\nc{\bfy}{\mathbf y}
\nc{\base}[1]{\bfone^{\otimes ({#1}+1)}} 
\nc{\Cat}{\mathbf{Cat}}

\nc{\detail}{\marginpar{\bf More detail}
    \noindent{\bf Need more detail!}
    \svp}
\nc{\Int}{\mathbf{Int}} \nc{\Mon}{\mathbf{Mon}}
\nc{\rbtm}{{shuffle }} \nc{\rbto}{{Rota-Baxter }}
\nc{\remarks}{\noindent{\bf Remarks: }} \nc{\Rings}{\mathbf{Rings}}
\nc{\Sets}{\mathbf{Sets}} \nc{\wtot}{\widetilde{\odot}}
\nc{\wast}{\widetilde{\ast}} \nc{\bodot}{\bar{\odot}}
\nc{\bast}{\bar{\ast}} \nc{\hodot}[1]{\odot^{#1}}
\nc{\hast}[1]{\ast^{#1}} \nc{\mal}{\mathcal{O}}
\nc{\tet}{\tilde{\ast}} \nc{\teot}{\tilde{\odot}}
\nc{\oex}{\overline{x}} \nc{\oey}{\overline{y}}
\nc{\oez}{\overline{z}} \nc{\oef}{\overline{f}}
\nc{\oea}{\overline{a}} \nc{\oeb}{\overline{b}}
\nc{\weast}[1]{\widetilde{\ast}^{#1}}
\nc{\weodot}[1]{\widetilde{\odot}^{#1}} \nc{\hstar}[1]{\star^{#1}}
\nc{\lae}{\langle} \nc{\rae}{\rangle}
\nc{\lf}{\lfloor}\nc{\rf}{\rfloor}

\nc{\QQ}{{\mathbb Q}}
\nc{\RR}{{\mathbb R}} \nc{\ZZ}{{\mathbb Z}}

\nc{\cala}{{\mathcal A}} \nc{\calb}{{\mathcal B}}
\nc{\calc}{{\mathcal C}}
\nc{\cald}{{\mathcal D}} \nc{\cale}{{\mathcal E}}
\nc{\calf}{{\mathcal F}} \nc{\calg}{{\mathcal G}}
\nc{\calh}{{\mathcal H}} \nc{\cali}{{\mathcal I}}
\nc{\call}{{\mathcal L}} \nc{\calm}{{\mathcal M}}
\nc{\caln}{{\mathcal N}} \nc{\calo}{{\mathcal O}}
\nc{\calp}{{\mathcal P}} \nc{\calr}{{\mathcal R}}
\nc{\cals}{{\mathcal S}} \nc{\calt}{{\mathcal T}}
\nc{\calu}{{\mathcal U}} \nc{\calw}{{\mathcal W}} \nc{\calk}{{\mathcal K}}
\nc{\calx}{{\mathcal X}} \nc{\CA}{\mathcal{A}}

\nc{\fraka}{{\mathfrak a}} \nc{\frakA}{{\mathfrak A}}
\nc{\frakb}{{\mathfrak b}} \nc{\frakB}{{\mathfrak B}}
\nc{\frakD}{{\mathfrak D}} \nc{\frakg}{{\mathfrak g}}
\nc{\frakH}{{\mathfrak H}} \nc{\frakL}{{\mathfrak L}}
\nc{\frakM}{{\mathfrak M}} \nc{\bfrakM}{\overline{\frakM}}
\nc{\frakm}{{\mathfrak m}} \nc{\frakP}{{\mathfrak P}}
\nc{\frakN}{{\mathfrak N}} \nc{\frakp}{{\mathfrak p}}
\nc{\frakS}{{\mathfrak S}}

\nc{\BS}{\mathbb{S
}}

\font\cyr=wncyr10 \font\cyrs=wncyr7
\nc{\li}[1]{\textcolor{red}{Li:#1}}
\nc{\jun}[1]{\textcolor{blue}{Jun: #1}}

\begin{document}

\title{Averaging algebras, Schr\"oder numbers, rooted trees and operads}
%
\author{Jun Pei}
\address{School of Mathematics and Statistics, Southwest University, Chongqing 400715, China}
         \email{peitsun@163.com}

\author{Li Guo}
\address{Department of Mathematics and Computer Science,
         Rutgers University,
         Newark, NJ 07102}
\email{liguo@rutgers.edu}

\date{\today}
\begin{abstract}
In this paper, we study averaging operators from an algebraic and combinatorial point of view. We first construct free averaging algebras in terms of a class of bracketed words called averaging words. We next apply this construction to obtain generating functions in one or two variables for subsets of averaging words when the averaging operator is taken to be idempotent. When the averaging algebra has an idempotent generator, the generating function in one variable is twice the generating function for large Schr\"oder numbers, leading us to give interpretations of large Schr\"oder numbers in terms of bracketed words and rooted trees, as well as a recursive formula for these numbers. We also give a representation of free averaging algebras by unreduced trees and apply it to give a combinatorial description of the operad of averaging algebras.
\end{abstract}

\maketitle

\tableofcontents

\setcounter{section}{0}

\section{Introduction}
Let $\bfk$ be a unitary commutative ring. An averaging operator on a commutative $\bfk$-algebra $R$ is a linear operator $P$ satisfying the identity
\begin{equation}\label{eq:aveqc}
P(f P(g)) = P(f)P(g) \text{ for all } f,g \in R.
\end{equation}
This operator was already implicitly studied by O. Reynolds~\cite{Re} in 1895 in turbulence theory under the disguise of a Reynolds operator which is defined by
\begin{equation}
P(f g) = P(f)P(g)+ P[(f-P(f))( g-P(g) )] \text{ for all } f,g \in R,
\mlabel{eq:rey}
\end{equation}
since an idempotent Reynolds operator is an averaging operator. An important class of such operators used in turbulence theory is the class of averages over one portion of space time of certain vector fields.
For example, the time average of a real valued function $f$ defined on time-space
\begin{equation}\notag 
f(x, t) \mapsto \bar{f}(x, t) = \lim_{T \rightarrow \infty} \frac{1}{2T}\int_{-T}^{T} f(x, t+\tau ) d \tau,
\end{equation}
is such an operator.

In the 1930s, the notion of averaging operator was explicitly defined by Kolmogoroff and Kamp\'{e} de F\'{e}riet~\cite{KF,Mil}.  Then G. Birkhoff \cite{Bi} continued its study and showed that a positive bounded projection in the Banach algebra $C(X)$, the algebra of scalar valued continuous functions on a compact Hausdorff space $X$, onto a fixed range space is an idempotent averaging operator. In 1954, S.~T.~C. Moy~\cite{STC} made the connection between averaging operators and conditional expectation. Furthermore, she studied the relationship between integration theory and averaging operators in turbulence theory and probability. Then her results were extended by G.~C. Rota \cite{R1}. During the same period, the idempotent averaging operators on $C_{\infty}(X)$, the algebra of all real valued continuous functions on a locally compact Hausdorff space $X$ that vanish at the infinity, were characterized by J.~L. Kelley \cite{Kel}.

Later on, more discoveries of averaging operators on various spaces were made. B. Brainerd \mcite{Bbr} considered the conditions under which an averaging operator can be represented as an integration on the abstract analogue of the ring of real valued measurable functions. In 1964, G.~C. Rota~\cite{R2} proved that a continuous Reynolds operator on the algebra $L_{\infty}(S,\Sigma,m)$ of bounded measurable functions on a measure space
$(S,\Sigma,m)$ is an averaging operator if and only if it has a closed range. J.~L.~B. Gamlen and J.~B. Miller \cite{Mil,GM} considered averaging operators on {\it noncommutative} Banach algebras where the averaging identities are defined by Eq.~(\ref{eq:avel}).
They discussed spectrum and resolvent sets of averaging operators on Banach algebras. N.~H. Bong~\mcite{Bo} found some connections between the resolvent of a Rota-Baxter operator~\mcite{Ba,Gub,Ro} and that of an averaging operator on complex Banach algebras. In 1986, Huijsmans generalized the work of Kelley to the case of $f$-algebras. Triki \cite{Tr1,Tr2} showed that a positive contractive projection on an Archimedean $f$-algebra is an idempotent averaging operator.

In the last century, most studies on averaging operators had been done for various special algebras, such as function spaces, Banach algebras, and the topics and methods had been largely analytic. In his Ph.~D. thesis in 2000~\cite{Cao}, W. Cao studied averaging operators in the general context and from an algebraic point of view. He gave the explicit construction of free commutative averaging algebras and studied the naturally induced Lie algebra structures from averaging operators.

In this century, while averaging operators continued to find many applications in its traditional areas of analysis and applied areas~\mcite{Fe}, their algebraic study has been deepened and generalized.
J.~L. Loday~\cite{Lo} defined the diassociative algebra as the enveloping algebra of the Leibniz algebra by analogy with the associative algebra as the enveloping algebra of the Lie algebra. M. Aguiar \cite{Ag} showed that a diassociative algebra can be derived from an averaging associative algebra by defining two new operations $x \dashv y := x P(y)$ and $x \vdash y: = P(x)y$. An analogue process gives a Leibniz algebra from an averaging Lie algebra by defining a new operation $\{x,y\} := [P(x),y]$ and derives a (left) permutative algebra from an averaging commutative associative algebra. In general, an averaging operator was defined on any binary operad and this kind of process was systematically studied in~\cite{PBGN} by relating the averaging actions to a special construction of binary operads called duplicators~\cite{GK2,PBGN2}.
Combining the averaging operators actions with the Rota-Baxter operators \cite{BBGN,PBGN2} actions, we obtained another connection between Rota-Baxter operators and averaging operators: the resulting algebraic structures given by the actions of the two operators are Koszul dual to each other.

These diverse applications and connections of the averaging algebra motivate us to carry out a further algebraic and combinatorial study of the averaging algebra in this paper. It is well known that in the category of any given algebraic structure, the free objects play a central role in studying the other objects. Further the combinatorial nature of the algebraic structure is often revealed by its free objects (see~\mcite{Re} for the Lie algebra case). Thus our first step is to construct free averaging algebras, after presenting some preliminary properties and examples of averaging algebras. This is carried out in Section~\mref{sec:freeav}, where the free averaging algebra on a set is realized on the free module on a set of bracketed words composed from the set, called averaging words. In Section \ref{sec:enu}, we begin our combinatorial investigation by enumerating subsets of averaging words for the free averaging algebra on one generator and when the operator is taken to be idempotent. The generating function from the enumeration of averaging words turns out to be twice the generating function of the large Schr\"oder numbers, revealing the combinatorial nature of averaging algebras. Pursuing this numerical connection of averaging algebra with large Schr\"oder numbers further allows us to find two applications of averaging algebras to large Schr\"oder numbers. We obtained two interpretations of large Schr\"oder numbers, one in terms of averaging words, another in terms of decorated rooted trees. We also obtain a recursive formula for large Schr\"oder numbers from such a formula arising from the study of averaging words. In Section~\mref{sec:tree}, we identify the set of averaging words with a special class of unreduced planar trees and applied it to give a combinatorial description of the operad of averaging algebras.

\section{Properties and free objects of averaging algebras}
\label{sec:freeav}
{\bf Convention. } Throughout this paper, all algebras are taken to be nonunitary unless otherwise specified.

In this section, we first give some properties and examples of averaging operators. We then give an explicit construction of the free averaging algebra on a non-empty set $X$.
\subsection{Definitions and properties}
\mlabel{ss:pre}
An averaging operator in the noncommutative context is given as follows.
\begin{defn}
A linear operator $P$ on a $\bfk$-algebra $R$ is called an {\bf averaging operator} if
\begin{equation}
P(x)P(y)=P(xP(y))=P(P(x)y) \text{ for all } x, y \in R.
\mlabel{eq:avel}
\end{equation}
A $\bfk$-algebra $R$ together with an averaging operator $P$ on $R$ is called an {\bf averaging algebra}.
\mlabel{de:av}
\end{defn}

It is well-known and easily checked that an idempotent operator is an averaging operator if and only if it is a Reynolds operator defined in Eq.~(\mref{eq:rey}).
There is also a close relationship between averaging operators and Rota-Baxter operators (of weight zero) on an algebra $R$. The latter operator is defined by the operator equation
$$ P(x)P(y)=P(P(x)y) + P(xP(y)) \text{ for all } x, y\in R$$
and has played important role in mathematics and physics~\mcite{Ba,Gub,Ro}.

Note that an averaging operator is a set operator in the sense that, for any semigroup $(G,\cdot)$, it makes sense to define an averaging operator on $G$ to be a map $P:G\to G$ such that
$$ P(x)\cdot P(y)=P(x\cdot P(y))=P(P(x)\cdot y) \text{ for all } x, y \in G.$$
Now let $P:\QQ[x]\to \QQ[x]$ be a linear operator such that, for each $n\geq 0$, we have $P(x^n)=\beta(n)x^{\theta(n)}$ with $\beta(n)\in \QQ$ and $\theta(n)>0$. It is shown in~\mcite{GRZ} $P$ that $P$ is a Rota-Baxter operator of weight zero
if and only if $\theta$ is an averaging operator on the additive semigroup $\ZZ_{\geq 1}$:
$$ \theta(m)+\theta(n)=\theta(m+\theta(n))=\theta(\theta(m)+n) \text{ for all } m, n\geq 0.$$

In addition to the examples of averaging algebras mentioned in the introduction, we display the following classes of examples.
We first give some examples from averaging processes.
Note that if $P$ is an averaging operator, then $c P$ is also an averaging operator for any $c\in \bfk$.

\begin{prop}
Let $R$ be a $\bfk$-algebra.
\begin{enumerate}
\item
Let $G$ be a finite group that acts on $R$ (on the right) and preserves the multiplication of $R$: $(xy)^g=x^g y^g$ for all $x, y\in R$ and $g\in G$. Then the linear operator
\begin{equation}
P: R\to R, \quad x \mapsto \sum_{g\in G} x^g \text{ for all }x\in R,
\mlabel{eq:avr}
\end{equation}
is an averaging operator.
\mlabel{it:avra}
\item
Let $a$ be a fixed element in the center of $R$. Define $P_{a}(x): =ax$ for all $x\in R$. Then $P_{a}$ is an averaging operator on $R$.
\mlabel{it:avrb}
\end{enumerate}
\mlabel{pp:avr}
\end{prop}
\begin{proof}
(\mref{it:avra}) For any $x, y\in R$, we have
$$ P(xP(y))=\sum_{h\in G} \left (x\sum_{g\in G} y^g\right)^h =\sum_{h\in G} x^h \sum_{g\in G} y^{gh} =\sum_{h\in G}x^h \sum_{g\in G}y^g=P(x)P(y).$$
We similarly have $P(P(x)y)=P(x)P(y)$.
\smallskip

\noindent
(\mref{it:avrb}) For any $x, y\in R$, we have
$$ P(xP(y))=a(x (ay))=P(x)P(y), \quad
P(P(x)y)=a((ax)y)=(ax)(ay)=P(x)P(y).$$
Thus $P$ is an averaging operator.
\end{proof}

As an application of Proposition~\mref{pp:avr}.(\mref{it:avra}), consider $F([a,b),\bfk)$, the $\bfk$-algebra of $\bfk$-valued functions on the interval $[a,b), a<b$. For a fixed positive integer $n$, define
$$P: F([a,b),\bfk)\to F([a,b),\bfk),\quad  f(x)\mapsto \sum_i f\left(x+\frac{i}{b-a}\right), $$
where the sum is over $i\in \ZZ$ such that $x+\frac{i}{b-a}$ is in $[a,b)$. Then $P$ is an averaging operator on $F([a,b),\RR)$. Here we take $G$ to be the cyclic group $\ZZ/n\ZZ$ acting on $[a,b)$ by permuting the $n$ subintervals $[a+\frac{i}{b-a}, a+\frac{i+1}{b-a}), 0\leq i\leq n-1$. This action induces an action of $\ZZ/n\ZZ$ on $F([a,b),\bfk)$ and hence Proposition~\mref{pp:avr}.(\mref{it:avra}) applies.
When $P$ is replaced by $\displaystyle{\frac{1}{n}P}$ which makes sense whenever $\bfk$ contains $\QQ$, then we obtain the usual averaging operator.

As a special case of Proposition~\mref{pp:avr}.(\mref{it:avrb}), let $G$ be a finite group and let $\bfk[G]$ be the group algebra. Then the $\bfk$-linear operator
$$ P: \bfk[G]\to \bfk[G], \quad g\mapsto \sum_{h\in G} hg = \left (\sum_{h\in G} h\right) g, \text{ for all }g\in G,$$
is an averaging operator since $\sum\limits_{h\in G}h$ is in the center of $\bfk[G]$.

There are many averaging operators that do not come from an averaging process.
A derivation on a $\bfk$-algebra $R$ is a linear operator $d:R\longrightarrow R$ such that
    $$ d(xy)=d(x)y+xd(y) \quad \text{ for all } x, y\in R.$$
It is immediately checked that a differential operator derivation $d$ with  $d^{2} =0$ is an averaging operator.

Birkhoff showed that an averaging operator on a unitary $\bfk$-algebra that preserves the identity $1_R$ must be idempotent: $P^2(x)=P(1_R\,P(x))=P(1_R)P(x)=P(x)$ for all $x\in R$~\mcite{R1}.
We next determine the conditions for an idempotent linear operator to be an averaging operator. Recall that there is a bijection
$$ \{\text{idempotent linear operators on } R\}\leftrightarrow
\{\text{linear decompositions }R=R_0\oplus R_1\}$$
such that $R_0=\im P$ and $R_1=\ker P$. The linear map $P$ corresponding to $R=R_0\oplus R_1$ is called the {\bf projection onto $R_0$ along $R_1$}.

\begin{prop}
Let $R$ be a $\bfk$-algebra and let $P:R\to R$ be an idempotent linear map. Let $R=R_0\oplus R_1$ be the corresponding linear decomposition. Then $P$ is an averaging operator if and only if
\begin{equation}
R_0R_0\subseteq R_0, \quad R_0R_1\subseteq R_1, \quad R_1R_0\subseteq R_1.
\mlabel{eq:grad}
\end{equation}
\mlabel{pp:super}
\end{prop}
\begin{proof}
For any $x, y\in R$, denote $x=x_0+x_1$ and $y=y_0+y_1$ with $x_i, y_i\in R_i, i=0, 1$.

Suppose $P$ is an averaging operator. Then from $P(R)=R_0$ and $P(x)P(y)=P(xP(y))$ we obtain $R_0R_0\subseteq R_0$.
Then we have
\begin{eqnarray*}
P(x)P(y)&=&x_0y_0, \\ P(P(x)y)&=&P(x_0y_0+x_0y_1)=P(x_0y_0)+P(x_0y_1)=x_0y_0+P(x_0y_1), \\
P(xP(y))&=&P(x_0y_0+x_1y_0)=P(x_0y_0)+P(x_1y_0)=x_0y_0+P(x_1y_0).
\end{eqnarray*}
Thus from Eq.~(\mref{eq:avel}) we obtain  $P(x_0y_1)=P(x_1y_0)=0$ for all $x_i,y_i\in R_i, i=0, 1$. Therefore Eq.~(\mref{eq:grad}) holds since $R_1=\ker P$ by the definition of $P$.

Conversely, suppose Eq.~(\mref{eq:grad}) holds. Then we have
$$ P(P(x)y)=P(x_0y_0+x_0y_1)=P(x_0y_0)+P(x_0y_1)=x_0y_0=P(x)P(y)$$
and similarly
$P(xP(y))=P(x)P(y)$ for all $x, y\in R$. Thus $P$ is an averaging operator.
\end{proof}

Recall that a $\bfk$-superalgebra is a $\bfk$-algebra $R$ with a $\bfk$-module decomposition $R = R_{0} \oplus R_{1}$ such that $R_{i}R_{j} \subseteq R_{i+j}$ where the subscripts are taken modulo 2.
\begin{coro}
\begin{enumerate}
\item
An idempotent algebra endomorphism $P:R\to R$ is an averaging operator. In particular, when $R$ is an augmented $\bfk$-algebra with the augmentation map $\varepsilon:R\to \bfk$, then $\varepsilon$ is an averaging operator regarded as a linear operator on $R$.
\mlabel{it:idem}
\item
Let $R = R_{0} \oplus R_{1}$ be a $\bfk$-superalgebra. Then the projection $P$ of $R$ to $R_0$ along $R_1$ is an averaging operator on $A$.
\mlabel{it:super}
\end{enumerate}
\end{coro}
\begin{proof}
(\mref{it:idem}) Let $R_0:=\im\, P$ and $R_1:=\ker P$. Then we have $R=R_0\oplus R_1$ and $P$ is the projection to $R_0$ along $R_1$. Since $R_1$ is an ideal of $R$, Eq.~(\mref{eq:grad}) holds. Hence $P$ is an averaging operator.
\smallskip

\noindent
(\mref{it:super}) This follows since $R_0$ and $R_1$ satisfies Eq.~(\mref{eq:grad}).
\end{proof}

\subsection{The construction of free averaging algebras}
Free commutative averaging algebras were constructed in~\cite{Cao}.

\begin{prop} $($\cite[Theorem~2.6]{Cao}$)$ Let $A$ be a unitary commutative $\bfk$-algebra and let $\mathbf{Sy}(A)$ denote the symmetric algebra on $A$. On the tensor product algebra $\mathfrak{A}:=\mathfrak{A}(A): = A \otimes \mathbf{Sy}(A)$, define the linear operator
    $$P: \mathfrak{A} \longrightarrow \mathfrak{A}, \quad
    P\left (\sum_{i} a_{i} \otimes s_{i}\right):= \sum_{i}1 \otimes (a_{i}s_{i}) \quad \text{ for all } a_i\in A,  s_i\in \mathbf{Sy}(A),
    $$
where $a_is_i$ is the product in $\mathbf{Sy}(A)$. Then $P$ is the free commutative averaging algebra on $A$. When $A$ is taken to be the polynomial algebra $\bfk[X]$ on a set $X$, then $\mathfrak{A}(\bfk[X])$ is the free commutative averaging algebra on $X$.
\mlabel{pp:freecom}
\end{prop}

We now construct free (noncommutative) averaging algebras. We carry out the construction in this subsection (Theorem~\mref{thm:free}) and provide the proof of the theorem in the next subsection.

\subsubsection{A basis of the free averaging algebra}
Recall~\mcite{Gop} that an operated semigroup (or a semigroup with an operator) is a
semigroup $U$ together with an operator $\alpha: U \rightarrow U$. Let $X$ be a given nonempty set. We will first obtain a linear basis of the free averaging algebra on $X$ from the free operated semigroup $\frakS(X)$ on $X$~\mcite{Gub,Gop}.

For any nonempty set $Y$, let $S(Y)$ be the free semigroup generated by $Y$ and $\lfloor Y \rfloor := \{ \lfloor y \rfloor ~|~ y \in Y \}$ be a replica of $Y$. Thus $\lfloor Y \rfloor$ is a set that is indexed by $Y$ but disjoint with $Y$.

Let $X$ be a nonempty set. Define a direct system as follows. Let
$$
\mathfrak{S}_{0}:= S(X), \quad \mathfrak{S}_{1}:=S(X \sqcup \lfloor \mathfrak{S}_{0} \rfloor ) = S (X \sqcup \lfloor S(X)\rfloor ),
$$
with the natural injection
$$
i_{0,1}: \mathfrak{S}_{0} = S(X) \hookrightarrow \mathfrak{S}_{1} = S(X \sqcup \lfloor \mathfrak{S}_{0} \rfloor).
$$
Inductively assuming that $\mathfrak{S}_{n-1}$ and $
i_{n-2,n-1}: \mathfrak{S}_{n-2} \hookrightarrow \mathfrak{S}_{n-1}$ have been obtained for $n \geq 2$, we define
$$
\mathfrak{S}_{n}:= S (X \sqcup \lfloor \mathfrak{S}_{n-1}\rfloor)
$$
and have the injection
$$
\lfloor \mathfrak{S}_{n-2} \rfloor  \hookrightarrow \lfloor \mathfrak{S}_{n-1} \rfloor.
$$
Then by the freeness of $\mathfrak{S}_{n-1} = S(X \sqcup \lfloor \mathfrak{S}_{n-2} \rfloor)$, we have
$$
\mathfrak{S}_{n-1} = S(X \sqcup \lfloor \mathfrak{S}_{n-2} \rfloor) \hookrightarrow S(X \sqcup \lfloor \mathfrak{S}_{n-1} \rfloor) = \mathfrak{S}_{n}.
$$
Finally, define $\mathfrak{S}(X) := \displaystyle \underrightarrow{\lim} \mathfrak{S}_{n}$ and define an operator on $\mathfrak{S}(X)$ by
$$
w \longmapsto \lc w \rc \quad \mbox{for all}~ w \in \mathfrak{S}(X).
$$
The operator will be used to define an averaging operator later.
Elements in $\mathfrak{S}(X)$ are called {\bf bracketed words} on $X$.
Define the depth of $w\in \frakS(X)$ to be
\begin{equation}
\dpt(w) := \min \{ n ~|~ w \in \frakS_{n}\}.
\mlabel{eq:depth}
\end{equation}

Taking the limit in $\mathfrak{S}_n = S(X \sqcup \lfloor \mathfrak{S}_{n-1} \rfloor)$, we obtain
\begin{equation}
\mathfrak{S}(X) = S(X \sqcup \lfloor \mathfrak{S}(X) \rfloor).
\mlabel{eq:lim}
\end{equation}
Thus every bracketed word has a unique decomposition, called the {\bf standard decomposition},
\begin{equation}\mlabel{eq:decom}
w =w_{1}w_{2} \cdots w_{b},
\end{equation}
where $w_{i}$ is in $X$ or $\lfloor \frakS(X) \rfloor$ for $i = 1, 2, \cdots, b$. Then we define $b=b(w)$ to be the {\bf breadth} of $w$. Elements of $X\sqcup \lc \frakS (X)\rc$ are called {\bf indecomposable}. Define the {\bf head index} $h(w)$ of $w$ to be $0$ (resp. $1$) if $w_{1}$ is in $X$ (resp. $\lfloor \frakS(X) \rfloor$). Similarly define the {\bf tail index} $t(w)$ of $w$ to be $0$ (resp. $1$) if $w_{b}$ is in $X$ (resp. $\lfloor \frakS(X) \rfloor$). If $w$ is indecomposable, then $h(w)=t(w)$, called the {\bf index} $\mathrm{id}(w)$ of $w$. Further, by combining strings of indecomposable factors in $w$ that are in $X$, we obtain the {\bf block decomposition} of $w$:
\begin{equation}
w=\omega_1\cdots \omega_r,
\mlabel{eq:block}
\end{equation}
where each $\omega_i, 1\leq i\leq r,$ is in either $S(X)$ or $\lc \frakS(X)\rc$.

For example, for $w=x\lc y\lc x\rc\rc xy\lc y\rc$, we have $\dpt(w)=2, b(w)=5, h(w)=0, t(w)=1$. Its block decomposition is $w=x\lc y\lc x\rc\rc (xy) \lc y\rc$. So $r=4$.

As is known~\mcite{EG1,Gub}, the free Rota-Baxter algebra on a set is defined on the free $\bfk$-module spanned by the set of {\bf Rota-Baxter words} $\mathcal{R}(X)\subseteq \mathfrak{S}(X)$ consisting of bracketed words that do not contain a subword of the form $\lc u\rc\lc v\rc$ where $u, v\in \mathfrak{S}$. It is natural to consider the averaging case in a similar way: Choose the set $\mathfrak{B}$ of bracketed words that do not contain a subword of the forms $\lc u\rc \lc v\rc$ and $\lc \lc u\rc v\rc$, where $u, v\in \mathfrak{S}$. Unfortunately, this restriction is not enough. For example, we have
$$\lc x \lc x\rc^{(2)}  \rc =\lc x\rc \lc x\rc^{(2)} = \lc \lc x\rc\,\lc x\rc\rc = \lc x \lc x \rc \rc^{(2)}$$
by the axiom of an averaging operator. Here $\lc \ \rc^{(n)}, n\geq 0,$ denotes the $n$-th iteration of the operator $\lc\ \rc$. Thus only one of the two elements $\lc x\lc x\rc^{(2)}\rc$ and $\lc x\lc x\rc\rc^{(2)}$ can be kept in a basis for the free averaging algebra. This motivates us to give the following definition.


\begin{defn}
Let $X$ be a set. A bracketed word $w\in \frakS(X)$ is called an {\bf averaging word} if $w$ does not contain any subword of the form $\lc u\rc\lc v\rc, \lc\lc u\rc v\rc$ or $\lc u\lc v\rc^{(2)}\rc$ for $u, v\in \frakS(X)$.
The set of averaging words on $X$ is denoted by $\AW=\AW(X)$. \mlabel{de:ave}
\end{defn}

For example, $\lc x\rc x, \lc x\lc x\rc\rc, \lc x \lc x \rc \rc^{(2)}$, $\lc x \lc x \rc^{(2)}   x \lc x \rc \rc$ are averaging words on $\{x\}$.

We will prove in Theorem~\mref{thm:free} that the free $\bfk$-module $\bfk \AW$ spanned by the set $\AW=\AW(X)$, equipped with a suitably defined multiplication and linear operator is the free averaging algebra on $X$. In order to carry out the construction and proof, we give the following recursive description of $\AW$.

For any nonempty subsets $G$, $H$ and $H'$ of $\mathfrak{S}(X)$, denote
\begin{eqnarray*}
\Lambda(G,H) &=:& (\sqcup_{r \geq 1}(G \lfloor H \rfloor)^{r}) \sqcup (\sqcup_{r \geq 1}(\lfloor H \rfloor G)^{r}) \sqcup (\sqcup_{r \geq 0}(G \lfloor H \rfloor)^{r}G) \sqcup (\sqcup_{r \geq 0}(\lfloor H \rfloor G)^{r} \lfloor H \rfloor),\\
\Lambda^{+}(G,H,H') &=:& (\sqcup_{r \geq 1}(G \lfloor H \rfloor)^{r-1} G\lc H' \rc) \sqcup (\sqcup_{r \geq 0}(G \lfloor H \rfloor)^{r}G),
\end{eqnarray*}
where, for a subset $T$ of $\mathfrak{S}(X)$, $T^{r} := \{t_{1} \cdots t_{r}~|~t_{i} \in T,
1\leq i \leq r\}$ and $T^{0}: = {\bf 1}$, the empty word.

We construct direct systems $\{\AW^{+}_{n} \}_{n \geq 0}$, $\{\widetilde{\AW}^{+}_{n} \}_{n \geq 0}$, $\{\AW_{n} \}_{n \geq 0}$ from $\mathfrak{S}(X)$ by the following recursions. First denote $\AW_{0} = \AW_{0}^{+} = \widetilde{\AW}_{0}^{+}=S(X)$. Then for $n \geq 0$, define
\begin{eqnarray}
&\AW_{n+1} = \Lambda(\AW_{0}, \widetilde{\AW}_{n}^{+}), \mlabel{eq:recur1}& \\
&\AW_{n+1}^{+} =\Lambda^{+}(\AW_{0}, \widetilde{\AW}_{n}^{+}, \AW_{n}^{+}), \quad \widetilde{\AW}_{n+1}^{+} = \AW_{n+1}^{+} \sqcup \lc \widetilde{\AW}_{n}^{+} \rc.&
\end{eqnarray}

We have the following properties on $\AW_{n}$, $\AW^{+}_{n}$ and $\widetilde{\AW}^{+}_{n}$.
\begin{prop}
For $n \geq 0$, we have
\begin{eqnarray}
\AW_{n} &\subseteq& \AW_{n+1}, \mlabel{eq:lim1}\\
\AW_{n}^{+} &\subseteq& \AW_{n+1}^{+},\mlabel{eq:lim2}\\
\widetilde{\AW}_{n}^{+} &\subseteq& \widetilde{\AW}_{n+1}^{+}.\mlabel{eq:lim3}
\end{eqnarray}
\end{prop}
\begin{proof}
We prove the inclusions by induction on $n$. When $n=0$, by definition, we have
$$
\AW_{0} \subseteq \AW_{1} , \quad \AW_{0}^{+} \subseteq \AW_{1}^{+}, \quad
\widetilde{\AW}_{0}^{+} \subseteq \widetilde{\AW}_{1}^{+}.
$$
Suppose that the inclusions in  Eqs.(\ref{eq:lim1})--(\ref{eq:lim3}) hold for $n = k \geq 0$, that is $\AW_{k} \subseteq \AW_{k+1}$, $\AW_{k}^{+} \subseteq \AW_{k+1}^{+}$ and $\widetilde{\AW}_{k}^{+} \subseteq \widetilde{\AW}_{k+1}^{+}$. Consider the case $n=k+1$. Then we immediately have
\begin{eqnarray*}
&\AW_{k+1} = \Lambda(\AW_{0}, \widetilde{\AW}_{k}^{+})  \subseteq  \Lambda(\AW_{0}, \widetilde{\AW}_{k+1}^{+}) = \AW_{k+2},& \\
&\AW_{k+1}^{+} = \Lambda^{+}(\AW_{0}, \widetilde{\AW}_{k}^{+}, \AW_{k}^{+})  \subseteq  \Lambda^{+}(\AW_{0}, \widetilde{\AW}_{k+1}^{+}, \AW_{k+1}^{+}) = \AW_{k+2}^{+},&\\
&\widetilde{\AW}_{k+1} = \AW_{k+1}^{+} \sqcup \lc \widetilde{\AW}_{k}^{+}\rc \subseteq \AW_{k+2}^{+} \sqcup \lc \widetilde{\AW}_{k+1}^{+}\rc=\widetilde{\AW}_{k+2}^{+}.&
\end{eqnarray*}
These complete the induction.
\end{proof}

Therefore we can take the direct systems
$$
\AW_\infty := \bigcup_{n \geq 0} \AW_{n} = \underrightarrow{\lim} ~\AW_{n} ,\quad
\AW^{+}:= \bigcup_{n \geq 0} \AW_{n}^{+}
= \underrightarrow{\lim} ~\AW_{n}^{+}, \quad \widetilde{\AW}^{+} := \bigcup_{n \geq 0} \widetilde{\AW}_{n}^{+} = \underrightarrow{\lim} ~\widetilde{\AW}_{n}^{+}.
$$

\begin{prop}
For a given set $X$, we have $\AW=\AW_\infty$\,.
\mlabel{pp:equiv}
\end{prop}
Because of the proposition, we will omit the notation $\AW_\infty$ in the rest of the paper.

\begin{proof}
For now we let $\AW_{(n)}, n\geq 0,$ denote the subset of $\AW$ consisting of its elements of depth less or equal to $n$: $\AW_{(n)}:=\AW\cap \frakS_n$. Then we have $\AW=\cup_{n\geq 0} \AW_{(n)}$. So we just need to verify
\begin{equation}
\AW_{(n)}=\AW_n \text{ for all } n\geq 0.
\mlabel{eq:equiv}
\end{equation}
We will prove it by induction on $n\geq 0$.

When $n=0$, there is nothing to prove since $\AW_{(0)}=\AW_0=S(X)$. Assume that Eq.~(\mref{eq:equiv}) has been verified for $n\leq k$ with $k\geq 0$ and consider the subsets $\AW_{(k+1)}$ and $\AW_{k+1}$.

Since $\AW_k:=\Lambda^+(\AW_0,\tilde{\AW}^+_{k-1})$ is contained in $\AW_{(k)}\subseteq \AW_{(k+1)},$ the subsets $\AW_k^+$ and $\tilde{\AW}_k^+$ are contained in $\AW_{(k+1)}$ by the induction hypothesis. Thus elements in these subsets do not contain elements of the forms excluded in the definition of $\AW$. Further elements in $\tilde{\AW}_k^+$ do not contain elements of the forms $\lc u\rc v$ and $u\lc v\rc^{(2)}$. Thus elements in $\lc \tilde{\AW}_k^+\rc$ do not contain elements of the form $\lc \lc u\rc v\rc$ and $\lc u\lc v\rc^{(2)}\rc$. Therefore by the definition of $\AW_{k+1}$ in Eq.~(\mref{eq:recur1}), elements of $\AW_{k+1}$ do not contain subwords of the forms excluded in the definition of $\AW$. Since $\AW_{k+1}$ is also contained in $\frakS_{k+1}$, we have $\AW_{k+1}\subseteq \AW_{(k+1)}$.

Conversely, since elements of $\AW_{(k+1)}$ have depth less or equal to $k+1$ and do not contain subwords of the form $\lc u\rc \lc v\rc$, by the induction hypothesis we have
$\AW_{(k+1)}\subseteq \Lambda(\AW_0,\AW_k)$. Since elements in $\AW_k$ are in brackets in $\Lambda(\AW_0,\AW_k)$ and elements of $\AW_{(k+1)}$ cannot contain elements of the forms $\lc \lc u\rc v\rc$ and $\lc u\lc r\rc^{(2)}\rc$, we have
$$\AW_{(k+1)}\subseteq \Lambda^+(\AW_0,\tilde{\AW}_k^+)=\AW_{k+1}.$$
This completes the induction.
\end{proof}

\begin{remark}
If $v$ is in $\lc \widetilde{\AW}^{+} \rc$, then there is unique $s\geq 1$ such that $v=\lc v' \rc^{(s)}$ with $v' \in \AW^{+}$. Thus if $v' = v'_{1} \cdots v'_{n}$ is in standard form, then $v'_{1} =x$ and $v'_{n}$ is either $x$ or $\lc \tilde{v}'_{n}\rc$ with $\tilde{v}'_{n} \in \AW^{+}$ when $n \geq 2$.
\mlabel{rem:stand}
\end{remark}

Taking the limit on both sides of Eq.~(\mref{eq:recur1}), we obtain
$\AW=\Lambda (\AW_0,\widetilde{\AW}^+)$. Thus in the block decomposition $w=\omega_1\cdots \omega_r$ of $w\in \AW$, the elements $\omega_1,\cdots, \omega_r$ are alternatively in $\AW_0=S(X)$ and $\lc \widetilde{\AW}^+\rc$. We show next that these are also sufficient conditions for $w$ to be in $\AW$.

\begin{lemma}
\begin{enumerate}
\item
Let $w=\omega_1\cdots \omega_r$ be the block decomposition of $w\in \frakS(X)$ in Eq.~(\ref{eq:block}). Then $w$ is an averaging word if and only if $w$ is in $\Lambda(X,\frakS(X))$ and each $\omega_i, 1\leq i\leq r,$ is an averaging word and $\mathrm{id}(w_i)\neq \mathrm{id}(w_{i+1})$ for $1\leq i\leq r-1$.
\mlabel{it:block}
\item
Let $w=w_1\cdots w_b$ be the standard decomposition of $w\in \frakS(X)$ in Eq.~(\mref{eq:decom}). If $w$ is an averaging word, then each $w_i, 1\leq i\leq b,$ is an averaging word. Further, if $w$ is an averaging word and $w=uv$ with $u, v\in\frakS(X)$, then $u$ and $v$ are averaging words.
\mlabel{it:st}
\end{enumerate}
\mlabel{lem:ind}
\end{lemma}

\begin{proof}
(\mref{it:block})
Let $w\in \frakS(X)$ with block decomposition $w=\omega_1\cdots\omega_r$ be an averaging word. Then by the definition of an averaging word, $w$ does not contain any of the subwords excluded in the definition of averaging words. Thus each $\omega_i, 1\leq i\leq r,$ does not contain these subwords and hence is an averaging word.

Conversely, let $w=\omega_1\cdots \omega_r$ be in $\Lambda(X,\frakS(X))$ and $\omega_1,\cdots,\omega_r$ be averaging words. Since $w$ is in $\frakS_k$ for some $k\geq 0$, we just need to prove that $w$ is in $\AW$ by induction on $k\geq 0$. When $k=0$, we have $\frakS_0=S(X)$ which is in $\AW$. Hence the claim holds. Assume that the claim holds for all $w\in \frakS_k$ where $k\geq 0$ and consider $w\in \frakS_{k+1}$. Thus $w$ is in $\Lambda(X,\frakS)$ and each $\omega_i, 1\leq i\leq r,$ is in $\AW_{k+1}$. The second condition means that each $\omega_i$ is either in $S(X)$ or in $\lc \widetilde{\AW}_k^+\rc$. Then the first condition means that $w$ is in $\Lambda(S(X),\widetilde{\AW}_k^+)=\AW_{k+1}$. This completes the induction.

\smallskip

\noindent
(\mref{it:st})
Both statements follow directly from the definition of averaging words.
\end{proof}

\subsubsection{Construction of the product and operator}
Let $X$ be a set and let $\bfk \AW$ to be the free $\bfk$-module generated by $\AW:=\AW(X)$. To define a multiplication $\diamond$ on $\bfk\AW$, we first define  $u\diamond v$ for two words $u$ and $v$ in $\AW$ by induction on the depth $\dpt(u)\geq 0$ of $u$ as follows.


If $\dpt(u)=0$, then $u$ is in $S(X)$ and the product $\diamond$ is the concatenation. Assume that $u \diamond v$ have been defined for all $u, v\in \AW$ with  $\dpt(u) \leq k$ where $k \geq 0$. Consider $u, v\in \AW$ with $\dpt(u)=k+1$. First consider the case when $u$ and $v$ are indecomposable, namely are in $X\sqcup \lc \AW\rc$. For $u, v\in \lc \AW\rc$, rewrite $u = \lc u \rc^{(s)}, v = \lc v' \rc^{(t)}$, where $s, t\geq 1$ while $u', v'$ are in $\AW^{+}$ as in Remark~\mref{rem:stand} and hence are not in $\lc \AW\rc$. Then define
\begin{equation}\mlabel{eq:pro}
u \diamond v =\left \{\begin{array}{ll} uv, & \text{if } u \text{ or } v \text{ is in } X, \\
 \lc u' \diamond \lc v' \rc \rc^{(s+t-1)}, & \mbox{if} ~u = \lc u' \rc^{(s)} \text{ and } v = \lc v' \rc^{(t)},
 \end{array} \right .
\end{equation}
where $u'\diamond \lc v' \rc$ is defined by the induction hypothesis since $\dpt(u')=\dpt(u)-s$ is less than $k+1$. Next consider the general case when $u$ and $v$ are in $\AW$. Let $u = u_{1} u_{2} \cdots u_{m}$, $v = v_{1} v_{2}
 \cdots v_{n}$ be their standard decompositions in Eq.~(\mref{eq:decom}). Then define
\begin{equation}
u \diamond v = u_{1} u_{2} \cdots u_{m-1} (u_{m} \diamond v_{1}) v_{2} \cdots v_{n},
\mlabel{eq:pro2}
\end{equation}
where $u_m\diamond v_1$ is the concatenation or as defined in Eq.~(\mref{eq:pro}).

For example, for $u_{1} = x, v_{1} = \lfloor x \rfloor$, $u_{2} = \lc x \lc x \rc \rc^{(2)}$ and $v_{2} = \lc x \rc^{(3)}$, we have
\begin{eqnarray*}
u_{1} \diamond v_{1} = x\lfloor x \rfloor, \quad u_{2} \diamond v_{2} = \lc x \lc x \rc \diamond \lc x\rc \rc^{(4)} = \lc x \lc x \lc x\rc\rc \rc^{(4)}.
\end{eqnarray*}
By the concatenation case and Eq.~(\mref{eq:pro}), we have
\begin{equation}
h(u) = h(u \diamond v), \qquad t(v) = t(u \diamond v).
\mlabel{eq:sp0}
\end{equation}

Extending $\diamond$ bilinearly, we obtain a binary operation on $\bfk\mathcal{A}(X)$. The following properties can be derived from the definition of $\diamond$ directly.
\begin{lemma}\label{lem:sp}
Let $w, w' \in \AW$. Then
\begin{enumerate}
\item $h(w) = h(w \diamond w')$ and $t(w') = t(w \diamond w')$.
\item If $t(w) \neq h(w')$ or $t(w) =h(w')=0$, then for any $w'' \in \AW$,
\begin{equation}\mlabel{eq:sp1}
(ww') \diamond w'' = w (w' \diamond w''),
\end{equation}
\begin{equation}\mlabel{eq:sp2}
w'' \diamond (ww') = (w'' \diamond w)w'.
\end{equation}
\end{enumerate}
\end{lemma}

\subsubsection{The construction of the operator}
We next define a linear operator $P$ on $\bfk \AW$. For $u \in \AW$, there is some $n$ such that $u \in \AW_{n}$. Recall that
$$ \AW^{+}_{n} = \big(\sqcup_{r \geq 1}(\AW_{0} \lfloor \widetilde{\AW}_{n-1}^{+} \rfloor)^{r-1} \AW_{0}\lc \AW^{+}_{n-1} \rc\big) \sqcup \big(\sqcup_{r \geq 0}(\AW_{0} \lfloor \widetilde{\AW}_{n-1}^{+} \rfloor)^{r}\AW_{0}\big),\quad
\widetilde{\AW}^{+}_{n}= \AW_{n}^{+} \sqcup \lfloor \widetilde{\AW}_{n-1}^{+} \rfloor
$$
and
\begin{eqnarray}
\AW_{n} &=&\big(\sqcup_{r \geq 1}(\AW_{0} \lfloor \widetilde{\AW}_{n-1}^{+} \rfloor)^{r}\big) \sqcup \big(\sqcup_{r \geq 1}(\lfloor \widetilde{\AW}_{n-1}^{+} \rfloor \AW_{0})^{r}\big) \mlabel{eq:odec}
\\
&&\sqcup \big(\sqcup_{r \geq 0}(\AW_{0} \lfloor \widetilde{\AW}_{n-1}^{+} \rfloor)^{r}\AW_{0}\big) \sqcup \big(\sqcup_{r \geq 0}(\lfloor \widetilde{\AW}_{n-1}^{+} \rfloor \AW_{0})^{r} \lfloor \widetilde{\AW}_{n-1}^{+} \rfloor\big).\notag
\end{eqnarray}
Since $\widetilde{\AW}_{n-1}^{+} = \AW_{n-1}^{+} \sqcup \lc \widetilde{\AW}_{n-2}^{+} \rc$, the first of the four disjoint union component becomes
$$
\sqcup_{r \geq 1}(\AW_{0} \lfloor \widetilde{\AW}_{n-1}^{+} \rfloor)^{r} = \big(\sqcup_{r \geq 1}(\AW_{0} \lfloor \widetilde{\AW}_{n-1}^{+} \rfloor)^{r-1} \AW_{0}\lc \AW^{+}_{n-1} \rc\big) \sqcup \big(\sqcup_{r \geq 1}(\AW_{0} \lfloor \widetilde{\AW}_{n-1}^{+} \rfloor)^{r-1} \AW_{0}\lc  \widetilde{\AW}^{+}_{n-2} \rc^{(2)}\big),
$$
of which the first disjoint component is exactly the first disjoint union component of $\AW_n^+$. Also the third disjoint union component of $\AW_n$ is the second disjoint union component of $\AW_n^+$.
By collecting the components of $\widetilde{\AW}_{n}^{+}$ together in this way, we see that $\AW_{n}$ can be rearranged as
\begin{equation}
\AW_{n} = \widetilde{\AW}_{n}^{+} \sqcup \big(\sqcup_{r \geq 1}(\AW_{0} \lfloor \widetilde{\AW}_{n-1}^{+} \rfloor)^{r-1} \AW_{0}\lc  \widetilde{\AW}^{+}_{n-2} \rc^{(2)}\big) \sqcup \big(\sqcup_{r \geq 1}(\lfloor \widetilde{\AW}_{n-1}^{+} \rfloor \AW_{0})^{r}\big) \sqcup \big(\sqcup_{r \geq 1}(\lfloor \widetilde{\AW}_{n-1}^{+} \rfloor \AW_{0})^{r} \lfloor \widetilde{\AW}_{n-1}^{+} \rc\big).
\mlabel{eq:odec2}
\end{equation}
So any $u\in \AW_n$ is in one of the above disjoint components. We accordingly define
$$
P_X(u) = \left\{\begin{array}{ll} \lfloor u \rfloor,  &\mbox{if}~ u \in \widetilde{\AW}^{+}_{n},\\
\lc u_{1} \diamond \lc u_{2} \rc \rc^{(s)}, & \mbox{if}~u \in \sqcup_{r \geq 1}(\lc \widetilde{\AW}_{n-1}^{+} \rc  \AW_{0} )^{r} \text{ with } u = \lc u_{1} \rc^{(s)} u_{2}, \\
\lc u_{1}' \lc u_{2}' \rc \rc^{(s)}, &\mbox{if}~u \in \sqcup_{r \geq 1}(\AW_{0} \lfloor \widetilde{\AW}_{n-1}^{+} \rfloor)^{(r)} \mbox{ with }~u = u_{1}'\lc u_{2}' \rc^{(s)}, s \geq 2,\\
\lc u_{1} \diamond \lc u_{2} \lc u_{3} \rc \rc  \rc^{(s+t-1)},&\mbox{if}~u \in \sqcup_{r \geq 1}(\lfloor \widetilde{\AW}_{n-1}^{+} \rfloor \AW_{0})^{r} \lfloor \widetilde{\AW}_{n-1}^{+} \rfloor \mbox{ with }~u = \lc u_{1} \rc^{(s)}u_{2} \lc u_{3}\rc^{(t)},
\end{array} \right.
$$
where $u_{1}, u_{2}', u_{3}$ are in $\AW^{+}$ and $h(u_{2}) =t(u_{2})=0$. Thus $P_X(u)$ is in $\AW$.
Then extending $P_X$ by linearity to a linear operator on $\mathcal{A}(X)$ that we still denote by $P_X$.

For example, if $u = \lfloor x \lfloor y \rfloor \rfloor z$, $v = x \lc y \rc^{(2)}$ and $w = \lc x \lc y \rc \rc z \lc x \rc^{(2)} $ then we have
$$
P_X(u) = \lfloor x\lfloor y \rfloor \diamond \lfloor z \rfloor  \rfloor = \lfloor x \lfloor y \lfloor z \rfloor \rfloor \rfloor, \quad P_X(v)=\lc x \lc y \rc  \rc^{(2)},
$$
$$
P_X(w) = \lc x\lc y \rc \diamond \lc z \lc x \rc  \rc    \rc^{(2)} = \lc x \lc y\lc z\lc x \rc \rc  \rc \rc^{(2)}.
$$

Let
$
j_{X}: X \longrightarrow S(X) \longrightarrow \AW \longrightarrow \bfk \AW
$
denote the natural injection from $X$ to $\mathcal{A}(X)$.
\begin{theorem}\label{thm:main}
Let $X$ be a non-empty set. Then
\begin{enumerate}
\item The pair $(\bfk \mathcal{A}, \diamond)$ is an algebra;
\mlabel{it:alg}
\item the triple $(\bfk\mathcal{A}, \diamond, P_X)$ is an averaging algebra;
\mlabel{it:ave}
\item the quadruple $(\bfk\mathcal{A}, \diamond, P_X, j_{X})$ is the free averaging algebra on set $X$. More precisely, for any averaging algebra $B$ and a map $f: X \longrightarrow B$, there is a unique averaging \bfk-algebra homomorphism $\bar{f}: \bfk\mathcal{A} \longrightarrow B$ such that $f = \bar{f} \circ j_{X}$.
\mlabel{it:free}
\end{enumerate}
\mlabel{thm:free}
\end{theorem}
The proof of the theorem is given in the next subsection.

\subsection{The proof of Theorem~\mref{thm:free}}
We now prove Theorem~\mref{thm:free}.

\noindent
(\mref{it:alg})
We only need to verify the associativity of $\diamond$:
\begin{equation}\mlabel{eq:asso}
(W \diamond W') \diamond W'' = W \diamond (W' \diamond W'') \quad \text{for all }  W,W',W'' \in \AW.
\end{equation}
For this we proceed by induction on the depth $\dpt(W)\geq 0$ of $W$. If $\dpt(W)=0$, then $W$ is in $S(X)$. Then $W\diamond W'=WW'$ and $W\diamond (W'\diamond W'')=W(W'\diamond W'')$ and the associativity follows from the definition of $\diamond$ in Eqs.~(\ref{eq:pro}) and (\mref{eq:pro2}). Suppose Eq.~(\mref{eq:asso}) has been verified for all $W, W', W''\in \AW$ with $\dpt(W)\leq k$ for $k\geq 0$ and consider $W, W', W''\in \AW$ with $\dpt(W)=k+1$.
We consider three cases.
\smallskip

\noindent
{\bf Case I. Suppose $t(W) \neq h(W')$ or $t(W) = h(W') = 0$: } Then by Lemma \mref{lem:sp}, we have
\begin{equation}\mlabel{eq:semiasso}
(W \diamond W') \diamond W'' = (WW') \diamond W'' = W(W' \diamond W'') = W \diamond (W' \diamond W'').
\end{equation}

\noindent
{\bf Case II. Suppose $t(W') \neq h(W'')$ or $t(W') = h(W'')=0$:} This case is proved similarly.
\smallskip

\noindent
{\bf Case III: Suppose $t(W) = h(W')=1$ and $t(W') = h(W'')=1$: } We divide this case into the following four subcases.
\begin{enumerate}
\item[(i)] {\bf Suppose $b(W') \geq 2$:} Then $W' = w_{1}'w_{2}'$ with $w_{1}', w_{2}' \in \AW$ and either $t(w_{1}') \neq h(w_{2}')$ or $t(w_{1}') = h(w_{2}') = 0$. Then we have
\begin{eqnarray*}
(W \diamond W')\diamond W'' &=& (W \diamond (w_{1}'w_{2}')) \diamond W'' \\
&=& ((W \diamond  w_{1}')w_{2}') \diamond W'' \quad \mbox{(by Eq. (\mref{eq:sp2}))}  \\
&=& (W \diamond w_{1}') (w_{2}' \diamond W'') \quad \mbox{(by Eq. (\mref{eq:sp1}))} \\
&=& W \diamond (w_{1}' \diamond (w_{2}' \diamond W''))~~ \mbox{(by Case I)} \\
&=& W \diamond ((w_{1}'w_{2}') \diamond W'')  \quad \mbox{(by Eq. (\mref{eq:semiasso}))} \\
&=& W \diamond (W' \diamond W'').
\end{eqnarray*}
\item[(ii)]
{\bf Suppose $b(W) \geq 2$:} Then $W = w_{1}w_{2}$ with $w_{1} \in \AW$, $b(w_{2})=1$ and either $t(w_{1}) \neq h(w_{2})$ or $t(w_{1}) = h(w_{2}) =0$. By Eq.~(\ref{eq:semiasso}), we have
\begin{eqnarray*}
&(W \diamond W') \diamond W'' = ((w_{1}w_{2}) \diamond W') \diamond W'' = (w_{1}(w_{2} \diamond W')) \diamond W'' = w_{1} ((w_{2} \diamond W') \diamond W'')&
\end{eqnarray*}
and
\begin{eqnarray*}
&W \diamond (W' \diamond W'') = (w_{1}w_{2}) \diamond (W' \diamond W'')= w_{1}(w_{2} \diamond (W' \diamond W'')).&
\end{eqnarray*}
Thus $(W \diamond W') \diamond W'' = W \diamond (W' \diamond W'')$ holds if and only if $
(w_{2} \diamond W' ) \diamond W'' = w_{2} \diamond (W' \diamond W'')$ holds. Therefore this case is reduced to the case when $b(W)=1$ in Subcase (iv).
\item[(iii)]
{\bf Suppose $b(W'') \geq 2$:} Then $W'' = w_{1}''w_{2}''$ with $w_{2}'' \in \AW$, $b(w_{1}'')=1$ and either $t(w_{1}'') \neq h(w_{2}'')$ or $t(w_{1}'') = h(w_{2}'') =0$.  Similarly to the case when $b(W) \geq 2$, we obtain $(W \diamond W') \diamond W'' = W \diamond (W' \diamond W'')$ holds if and only if $
(W \diamond W' ) \diamond w_{1}'' = W \diamond (W' \diamond w_{1}'')$ holds. Thus this case is also reduced to the case when $b(W'')=1$ in Subcase (iv).

\smallskip
In summary we have reduced the proof of Case III to the proof of the following special case:

\item[(iv)]
{\bf Suppose $b(W) = b(W') = b(W'') =1$:} Then all the three words are in $\lfloor \widetilde{\AW}^{+} \rfloor$.
Then $W = \lc w \rc^{(r)}, W' = \lc w' \rc^{(s)}, W'' = \lc w'' \rc^{(t)}$, where $r, s, t\geq 1$, $w, w', w'' \in \AW^{+}$ and $\dpt(W) =k+1$.
We have
\begin{eqnarray*}
&(W \diamond W') \diamond W'' = (\lc w \diamond \lc w' \rc \rc^{(r+s-1)}) \diamond \lc w'' \rc^{(t)} = \lc (w \diamond \lc w' \rc) \diamond \lc w'' \rc \rc^{(r+s+t-2)},&
\end{eqnarray*}
\begin{eqnarray*}
&W \diamond (W' \diamond W'') = \lc w \rc^{(r)} \diamond  (\lc w' \diamond \lc w'' \rc \rc^{(s+t-1)})= \lfloor  w \diamond (\lfloor w' \diamond \lfloor w'' \rfloor \rfloor) \rc^{(r+s+t-2)}.&
\end{eqnarray*}
Since $\dpt(w) = k+1-s$, by the induction hypothesis we have
$$
(w \diamond \lfloor w' \rfloor) \diamond \lfloor w'' \rfloor   = w \diamond (\lfloor w' \rfloor  \diamond \lfloor w'' \rfloor) = w \diamond (\lfloor w' \diamond \lfloor w'' \rfloor \rfloor)
$$
and then $(W \diamond W')\diamond W'' = W \diamond (W' \diamond W'')$.
\end{enumerate}
This completes the inductive proof of Eq.~(\mref{eq:asso}).
\smallskip

\noindent
(\mref{it:ave})  We only need to verify the equations
\begin{equation}
P_X(u)\diamond P_X(v)  =P_X (P_X(u) \diamond v),\quad P_X(u)\diamond P_X(v) = P_X(u\diamond P_X(v)) \quad \text{ for all } u,v \in \AW.
\mlabel{eq:op}
\end{equation}
We first recall by Remark~\mref{rem:stand} that for any $u,v \in \AW$, there exist unique $u', v' \in \AW^{+}$ and $s, t\geq 1$ such that \begin{equation}
P_X(u) = \lc u' \rc^{(s)}, P_X(v) = \lc v' \rc^{(t)}
\mlabel{eq:pow}
\end{equation}
Then we have
\begin{equation}
P_X(u) \diamond P_X(v) = \lc u' \rc^{(s)} \diamond \lc v' \rc^{(t)} = \lc u' \diamond \lc v' \rc  \rc^{(s+t-1)}.
\mlabel{eq:pp}
\end{equation}

We verify the first equation in Eq.~(\mref{eq:op}) by considering the two cases when $h(v)=0, 1$.
\smallskip

\noindent
{\bf Case 1. Suppose $h(v) =0$:} When $t(v) =0$, we have $v \in \AW^{+}$. Then $P_X(v) = \lc v \rc=\lc v'\rc^{(t)}$ in Eq.~(\mref{eq:pow}) with $v'=v$ and $t=1$. When $t(v) =1$, we rewrite $v$ as $v_{1} \lc v_{2}\rc^{(\ell)}$, where $v_{2} \in \AW^{+}$. Then $P_X(v) = \lc v_{1} \lc v_{2}\rc \rc^{(\ell)}=\lc v'\rc^{(t)}$ with $v'  = v_{1} \lc v_{2} \rc$ and $t =\ell$. Then by Eq.~(\mref{eq:pp}), we have
\begin{equation*}
P_X(P_X(u) \diamond v) = P_X(\lc u' \rc^{(s)} v) =\left\{\begin{array}{ll} \lc u' \diamond \lc v \rc \rc^{(s)}, & t(v)=0 \\ \lc u' \diamond \lc v_{1} \lc v_{2} \rc  \rc \rc^{(s+\ell-1)}, & t(v)=1
\end{array}\right. = \lc u' \diamond \lc v' \rc  \rc^{(s+t-1)} =  P_X(u) \diamond P_X(v).
\end{equation*}

\noindent
{\bf Case 2. Suppose $h(v) =1$:} When $t(v)=0$, we rewrite $v$ as $\lc v_{1} \rc^{(\ell)}v_{2}$, where $v_{1} \in \AW^{+}$. Then $P_X(v) = \lc v_{1} \diamond \lc v_{2} \rc \rc^{(\ell)}=\lc v'\rc^{(t)}$ with $v' = v_{1} \diamond \lc v_{2} \rc$ and $t = \ell$. When $t(v)=1$, we rewrite $v$  as $\lc v_{1} \rc^{(\ell)} v_{2} \lc v_{3} \rc^{(q)}$. Then $P_X(v) = \lc v_{1} \diamond \lc v_{2} \lc v_{3} \rc  \rc  \rc^{(\ell + q-1)}=\lc v'\rc^{(t)}$ with $v' = v_{1} \diamond \lc v_{2} \lc v_{3} \rc  \rc $ and $t = \ell +q -1$. Then by Eq.~(\mref{eq:pp}), we have
\allowdisplaybreaks
    \begin{eqnarray*}
   P_X(P_X(u) \diamond v) &=& P_X(\lc u' \rc^{(s)} \diamond v)\\
   & =& \left\{
    \begin{array}{ll} P_X(\lc u' \diamond \lc v_{1} \rc \rc^{(s+\ell-1)} v_{2} ),& t(v)=0\\
    P_X(\lc u' \diamond \lc v_{1} \rc \rc^{(s+\ell-1)} v_{2}\lc v_{3} \rc^{(q)} ), & t(v)=1 \end{array}\right.\\
    &=& \left\{
    \begin{array}{ll}\lc (u' \diamond \lc v_{1} \rc) \diamond \lc v_{2}\rc \rc^{(s+\ell-1)}, & t(v)=0 \\
    \lc(u' \diamond \lc v_{1} \rc) \diamond \lc v_{2} \lc v_{3} \rc  \rc \rc^{(s+\ell+q-2)}, & t(v)=1 \end{array}\right. \\
&=&\left\{
    \begin{array}{ll} \lc u' \diamond \lc v_{1} \diamond \lc v_{2}\rc \rc \rc^{(s+\ell-1)}, & t(v)=0 \\
    \lc u' \diamond \lc v_{1} \diamond \lc v_{2} \lc v_{3} \rc \rc  \rc \rc^{(s+\ell+q-2)}, & t(v)=1\end{array}\right.\\
     &=& \lc u' \diamond \lc v' \rc  \rc^{(s+t-1)}\\
      &=& P_X(u) \diamond P_X(v).
    \end{eqnarray*}
Thus the first equation in Eq.~(\mref{eq:op}) is verified.

To verify the second equation in Eq.~(\mref{eq:op}), we also consider the two cases when $h(u)$ is $0$ and $1$.
\smallskip

\noindent
{\bf Case 1. Suppose $h(u) =0$:} When $t(u) =0$, we have $u \in \AW^{+}$. So $P_X(u) = \lc u \rc=\lc u'\rc^{(s)}$ as in Eq.~(\mref{eq:pow}) with  $u'=u$ and $s=1$. When $t(u) =1$, we rewrite $u$ as $u_{1} \lc u_{2}\rc^{(\ell)}$, where $u_{2} \in \AW^{+}$. Then $P_X(u) = \lc u_{1} \lc u_{2}\rc \rc^{(\ell)}=\lc u'\rc ^s$ with $u'  = u_{1} \lc u_{2} \rc, s =\ell$. Then by Eq.~(\mref{eq:pp}) we have

\begin{eqnarray*}
P_X(u \diamond P_X(v)) &=& P_X(u \lc v' \rc^{(t)}) \\ &=&\left\{\begin{array}{ll} \lc u \diamond \lc v' \rc \rc^{(t)}, & t(u)=0 \\ P_X( u_{1} \lc u_{2} \diamond \lc v' \rc \rc^{(\ell+t-1)}), &t(u)=1
\end{array}\right. \\
&=& \left\{\begin{array}{ll} \lc u \diamond \lc v' \rc \rc^{(t)}, &t(u)=0 \\ \lc  u_{1} \lc u_{2} \diamond \lc v' \rc \rc \rc^{(\ell+t-1)}, &t(u)=1
\end{array}\right.\\
&=& \left\{\begin{array}{ll} \lc u \diamond \lc v' \rc \rc^{(t)}, & t(u)=0 \\
\lc  u_{1} \lc u_{2} \rc \diamond \lc v' \rc \rc^{(\ell+t-1)}, & t(u)=1
\end{array}\right. \\
&=&  \lc u' \diamond \lc v' \rc  \rc^{(s+t-1)} \\
&=&  P_X(u) \diamond P_X(v).
\end{eqnarray*}

\noindent
{\bf Case 2. Suppose $h(u) =1$:} When $t(u)=0$, we rewrite $u$ as $\lc u_{1} \rc^{(\ell)}u_{2}$, $u_{1} \in \AW^{+}$. Then $P_X(u) = \lc u_{1} \diamond \lc u_{2} \rc \rc^{(\ell)}=\lc u'\rc^{(s)}$ as in Eq.~(\mref{eq:pow}) with $u' = u_{1} \diamond \lc u_{2} \rc$ and $s = \ell$. When $t(u)=1$, we rewrite $u$  as $\lc u_{1} \rc^{(\ell)} u_{2} \lc u_{3} \rc^{(q)}$. Then $P_X(u) = \lc u_{1} \diamond \lc u_{2} \lc u_{3} \rc  \rc  \rc^{(\ell + q-1)}=\lc u'\rc^{(s)}$ with $u' = u_{1} \diamond \lc u_{2} \lc u_{3} \rc  \rc$ and $s = \ell +q -1$. Then we have
    \begin{eqnarray*}
    P_X(u \diamond P_X(v)) &=& P_X(u \diamond \lc v' \rc^{(t)})\\
    & =& \left\{
    \begin{array}{ll} P_X(\lc u_{1} \rc^{(\ell)}u_{2} \lc v' \rc^{(t)}), & t(u)=0\\
    P_X(\lc u_{1} \rc^{(\ell)}u_{2}\lc u_{3} \diamond \lc v' \rc  \rc^{(q+t-1)}), & t(u)=1\end{array}\right. \\
    &=& \left\{
    \begin{array}{ll} \lc u_{1} \diamond \lc u_{2} \lc v' \rc \rc^{(\ell+t-1)}, & t(u)=0 \\
    \lc  u_{1} \diamond \lc u_{2} \lc u_{3} \diamond \lc v' \rc  \rc   \rc    \rc^{(\ell+q+t-2)}, & t(u)=1 \end{array}\right.\\
    &=& \left\{\begin{array}{ll} \lc (u_{1} \diamond \lc u_{2}\rc) \diamond \lc v' \rc \rc^{(\ell+t-1)}, &t(u)=0 \\
\lc (u_{1} \diamond \lc u_{2} \lc u_{3} \rc  \rc)  \diamond \lc v' \rc   \rc^{(\ell+q+t-1)}, & t(u)=1 \end{array} \right. \\
&=& \lc u' \diamond \lc v' \rc  \rc^{(s+t-1)} \\
&=&  P_X(u) \diamond P_X(v).
    \end{eqnarray*}
This completes the proof of the second equation in Eq.~(\mref{eq:op}). Therefore $P_X$ is an averaging operator.
\smallskip

\noindent
(\mref{it:free})  Let $(B,Q)$ be an averaging algebra and $\ast$ be the product in $B$. Let $f: X \longrightarrow B$ be a map. We will construct a $\bfk$-linear map $\bar{f}: \bfk \AW \longrightarrow B$ by defining $\bar{f}(w)$ for $w \in \AW$. We achieve this by applying the induction
on $n$ for $w \in \AW_{n}$. For $w=x_{1}x_{2}\cdots x_{m}\in \AW_{0} = S(X)$, where $x_i\in X, 1\leq i\leq m$, define $\bar{f}(w) = f(x_{1})\ast f(x_{2}) \ast \cdots \ast f(x_{m})$. Suppose $\bar{f}(w)$ has been defined for $w \in \AW_{n}$ and consider $w \in \AW_{n+1}$ which is defined by

$$
\AW_{n+1} =\big(\sqcup_{r \geq 1}(\AW_{0} \lfloor \widetilde{\AW}_{n}^{+} \rfloor)^{r}\big) \sqcup \big(\sqcup_{r \geq 1}(\lfloor \widetilde{\AW}_{n}^{+} \rfloor \AW_{0})^{r}\big) \sqcup \big(\sqcup_{r \geq 0}(\AW_{0} \lfloor \widetilde{\AW}_{n}^{+} \rfloor)^{r}\AW_{0}\big) \sqcup \big(\sqcup_{r \geq 0}(\lfloor \widetilde{\AW}_{n}^{+} \rfloor \AW_{0})^{r} \lfloor \widetilde{\AW}_{n}^{+} \rfloor\big).
$$
If $w$ is in $\sqcup_{r \geq 1}(\AW_{0} \lfloor \widetilde{\AW}_{n}^{+} \rfloor)^{r}$, then $
w = \displaystyle \prod_{i=1}^{r} (w_{2i-1} \lfloor w_{2i} \rfloor)$,
where $w_{2i-1} \in \AW_{0}$ and $w_{2i} \in \widetilde{\AW}_{n}^{+}$. By the construction of the multiplication $\diamond$
and the averaging operator $P_X$, we also can express it by
$$
w = \diamond_{i=1}^{r} (w_{2i-1} \diamond P_X(w_{2i})).
$$
Thus there is only one possible way to define $\bar{f}(w)$ in order for $\bar{f}$ to be an averaging
homomorphism:
\begin{equation}\mlabel{eq:freemap}
\bar{f}(w) = \ast_{i=1}^{r} (\bar{f}(w_{2i-1}) \ast Q(\bar{f}(w_{2i}))).
\end{equation}
$\bar{f}(w)$ can be similarly defined if $w$ is in the other unions. This proves the existence of $f$ as a map and its uniqueness.

We next prove that the map $\bar{f}$ defined in Eq.(\ref{eq:freemap}) is indeed an averaging algebra
homomorphism. We will first show that $\bar{f}$ is an algebra homomorphism, that is, for any $W, W' \in  \AW$,
\begin{equation}\mlabel{eq:homo}
\bar{f}(W \diamond W') = \bar{f}(W) \ast \bar{f}(W').
\end{equation}
We will prove Eq. (\mref{eq:homo}) by induction on $\mathbf{b}:= b(W)+b(W')\geq 2.$

If $b(W)+b(W')=2$, then $b(W) = b(W') =1$. We then prove Eq.~(\mref{eq:homo}) by induction on $\dpt(W)$. When $\dpt(W)=0$, that is $W \in S(X)$. Then $W \diamond W'$ is the concatenation and, by definition,
$$
\bar{f}(W \diamond W') = \bar{f}(W) \ast \bar{f}(W').
$$
Suppose Eq. (\mref{eq:homo}) holds for $W, W'$ with $0 \leq \dpt(W) \leq k$ and consider $W,W' \in \AW$ with $\dpt(W) = k+1$. If $\dpt(W')=0$, a similar argument as in the case of $\dpt(W) =0$ proves Eq. (\mref{eq:homo}). If $\dpt(W')>0$, then $W = \lc w \rc^{(s)}$ and $W' = \lc w' \rc^{(t)}$ with $w, w' \in \AW^{+}$. Hence
\begin{eqnarray*}
\bar{f}(W \diamond W') &=& \bar{f}(\lc w \diamond \lc w' \rc \rc^{(s+t-1)}) \\
&=& Q^{s+t-1} (\bar{f}(w \diamond \lc w' \rc)) \quad \mbox{(by definition)} \\
&=& Q^{s+t-1} (\bar{f}(w) \ast \bar{f}(\lfloor w' \rfloor))) \quad \mbox{(by induction hypothesis)} \\
&=& Q^{s+t-1}(\bar{f}(w) \ast Q (\bar{f}(w'))) \quad \mbox{(by definition)} \\
&=&Q^{s}(\bar{f}(w)) \ast Q^{t}(\bar{f}(w'))  \quad \mbox{($Q$ is an averaging operator)} \\
&=& \bar{f}(\lfloor w \rc^{(s)})) \ast \bar{f}(\lfloor w' \rc^{(t)})) \quad \mbox{(by definition)}\\
&=& \bar{f}(W) \ast \bar{f}(W').
\end{eqnarray*}

Now assume that Eq. (\ref{eq:homo}) holds
for all $W,W'\in \AW$ with $2 \leq \mathbf{b} \leq k$. Consider $W, W'\in \AW$ with $\mathbf{b} = k + 1$. If $t(W) \neq h(W')$ or $t(W) = h(W')=0$, then
$W \diamond W'$ is the concatenation. So we have
$$
\bar{f}(W \diamond W') = \bar{f}(WW') = \bar{f}(W) \ast \bar{f}(W')
$$
by Eq.~(\mref{eq:freemap}). Now let $t(W) = h(W')=1$. Since $k \geq 2$, we have $k + 1 \geq 3$, so at least one of $W$ and
$W'$ have breadth $\geq 2$. Without loss of generality, we can assume $b(W) \geq 2$. Then $W$ can be written as $w_{1}w_{2}$ with $b(w_{2}) = 1$. Thus we have $ t(w_{2})  = t(W) = h(W')=1$, $t(w_{1}) =0, h(w_{2}) =1$, and $b(w_{2})+b(W') < b(W)+b(W') = k+1$. Applying the associativity of the product and the induction
hypothesis, we have
\begin{eqnarray*}
\bar{f}(W \diamond W') &=& \bar{f}((w_{1}w_{2}) \diamond W') \\
&=& \bar{f}(w_{1} (w_{2} \diamond W'))\\
&=& \bar{f}(w_{1} \diamond (w_{2} \diamond W'))\\
&=& \bar{f}(w_{1}) \ast \bar{f}(w_{2} \diamond W') \quad (t(w_{1}) \neq h(w_{2})) \\
&=& \bar{f}(w_{1}) \ast (\bar{f}(w_{2}) \ast \bar{f} (W')) \quad(\mbox{by induction hypothesis})\\
&=& (\bar{f}(w_{1}) \ast \bar{f}(w_{2})) \ast \bar{f} (W') \\
&=& \bar{f} (w_{1} \diamond w_{2}) \ast \bar{f}(W')  \quad (\dpt(w_{1})+\dpt(w_{2}) \leq k )\\
&=&  \bar{f} (W) \ast \bar{f} (W').
\end{eqnarray*}

Now we are left to show that, for any word $w \in \AW$, we have
\begin{equation}\mlabel{eq:hom}
\bar{f}(P_X(w)) = Q(\bar{f}(w)).
\end{equation}

Corresponding to the four components of the disjoint union decomposition in Eq.~(\mref{eq:odec2}), we have the following four cases to consider.

If $w$ is in $\widetilde{\AW}^{+}$, then by definition $P_X(w) = \lc w \rc \in \AW$. By definition (Eq. (\ref{eq:freemap})) we have
\begin{equation*}
\bar{f}(P_X(w)) = \bar{f}(\lfloor w \rfloor) = Q(\bar{f}(w)).
\end{equation*}

If $w$ is in $\sqcup_{r \geq 1}(\AW_{0} \lfloor \widetilde{\AW}_{n-1}^{+} \rfloor)^{r-1} \AW_{0}\lc  \widetilde{\AW}^{+}_{n-2} \rc^{(2)}$, then write $w = w_{1}\lc w_{2} \rc^{(s)}, s \geq 2$ with $w_{2} \in \AW^{+}$. Then $P_X(w)= \lc w_{1} \lc w_{2} \rc \rc^{(s)}$. We have
\begin{equation*}
\bar{f}(P_X(w)) = \bar{f} (  \lc w_{1} \lc w_{2} \rc \rc^{(s)}) =  Q^{s} (\bar{f}   (w_{1} \lc w_{2} \rc   )) = Q^{s} (\bar{f}(w_{1}) \ast  Q(\bar{f}(\lc w_{2}\rc) ).
\end{equation*}
Since $Q$ is an averaging operator, we further have
$$\bar{f}(P_X(w)) =Q(\bar{f}(w_{1}) \ast Q^{s}(\bar{f}(w_{2}))) =Q (\bar{f}(w_{1}) \ast \bar{f}(\lc w_{2} \rc^{(s)})) = Q(\bar{f}(w)).
$$

If $w$ is in $\sqcup_{r \geq 1}(\lc \widetilde{\AW}^{+} \rc  \AW_{0} )^{r}$, then write $w = \lc w_{1} \rc^{(s)} w_{2}$ with $w_{1} \in \AW^{+}$. Then by definition $P_X(w)=\lc w_{1} \diamond \lc w_{2} \rc \rc^{(s)}$. Thus
\begin{equation*}
\bar{f}(P_X(w)) = \bar{f} ( \lc w_{1}  \diamond \lc w_{2} \rc \rc^{(s)})= Q^{s} (\bar{f} ( w_{1}  \diamond \lc w_{2} \rc  ) ) = Q^{s} (\bar{f}(w_{1}) \ast  \bar{f}( \lfloor w_{2} \rfloor )  )
= Q^{s} (\bar{f}(w_{1}) \ast Q (\bar{f}(w_{2}))).
\end{equation*}
Since $Q$ is an averaging operator, we further have
$$\bar{f}(P_X(w))= Q ( Q^{s}( \bar{f} (w_{1}) ) \ast \bar{f}(w_{2}))= Q (\bar{f}(\lfloor w_{1} \rc^{(s)} w_{2} ) ) =Q (\bar{f}(w)).
$$

If $w$ is in $\sqcup_{r \geq 1}(\lfloor \widetilde{\AW}^{+} \rfloor \AW_{0})^{r} \lfloor \widetilde{\AW}^{+} \rfloor$, then write $w = \lc w_{1} \rc^{(s)}w_{2} \lc w_{3}\rc^{(t)}$ with $w_{1},w_{2},w_{3} \in \AW^{+}$. Then $P_X(w) = \lc w_{1} \diamond \lc w_{2} \lc w_{3} \rc \rc  \rc^{(s+t-1)}$. We have
\begin{eqnarray*}
\bar{f}(P_X(w)) &=& \bar{f} (\lc w_{1} \diamond \lc w_{2} \lc w_{3} \rc \rc  \rc^{(s+t-1)}) \\
&=& Q^{s+t-1} (\bar{f}(  w_{1} \diamond \lc w_{2} \lc w_{3} \rc \rc     )   ) \\
&=&Q^{s+t-1}(\bar{f}(w_{1}) \ast  \bar{f}( \lc w_{2} \lc w_{3} \rc \rc     )  ) \\
&=& Q^{s+t-1} (\bar{f}(w_{1}) \ast   Q(\bar{f}(w_{2}\lc w_{3} \rc))  )\\
&=& Q^{s} (\bar{f}(w_{1}) \ast   Q^{t}(\bar{f}(w_{2}\lc w_{3} \rc))) \quad \mbox{($Q$ is an averaging operator)}\\
&=& Q^{s} (\bar{f}(w_{1}) \ast   Q^{t}(\bar{f}(w_{2}) \ast Q(\bar{f}(w_{3})))) \\
&=& Q^{s} (\bar{f}(w_{1}) \ast   Q\big(\bar{f}(w_{2}) \ast Q^{t}(\bar{f}(w_{3}))\big)) \quad \mbox{($Q$ is an averaging operator)}\\
&=& Q(Q^{s}(\bar{f}(w_{1})) \ast \big(\bar{f}(w_{2}) \ast Q^{t}(\bar{f}(w_{3}))\big)) \quad \mbox{($Q$ is an averaging operator)}\\
&=& Q (\bar{f}(\lc w_{1} \rc^{(s)}) \ast (\bar{f}(w_{2}) \ast \bar{f}(\lc w_{3} \rc^{(t)})))\\
&=&Q(\bar{f}(w)).
\end{eqnarray*}

Therefore $\bar{f}$ commutes with the averaging operator. This completes the proof of Theorem~\mref{thm:free}.(\mref{it:free}).

The proof of Theorem~\mref{thm:free} is now completed.

\section{Enumeration in averaging algebras and large Schr\"oder numbers}
\mlabel{sec:enu}
In this section, we study the enumeration and generating functions of free averaging algebras. We first give the generating function of averaging words in two variables parameterizing the number of appearances of the variable and the operator respectively. We then observe that the generating function in one variable for averaging words (resp. indecomposable averaging words) with one idempotent operator and one idempotent generator is twice (resp. $z$ times) the generating function of large Schr\"oder numbers. This motivates us to give two interpretations of large Schr\"oder numbers in terms of averaging words and a class of decorated rooted trees. As a result, we obtain a recursive formula for large Schr\"oder numbers.

\subsection{Enumeration of averaging words}
In this section, we restrict ourselves to the set of averaging words with one generator and one idempotent operator and then give some results on enumerations of this set. For an idempotent operator, $P$ is an averaging operator if and only if it is a Reynolds operator. Interestingly, in most applications of averaging algebras in physics (invariant theory and fluid dynamics), function spaces, Banach algebras, the operators are idempotent.

Under the condition that the operator $\lc\ \rc$ is idempotent, no two pairs of brackets can be immediately adjacent or nested in an averaging word. Enumerations of Rota-Baxter words are given in~\mcite{GS}. We will follow the similar notations and apply the similar method to solve the enumeration problem of averaging words.

For an averaging word $w$, an {\bf $x$-run} is any occurrence in $w$ of consecutive products of $x$ of maximal length.
Let $v$ be either
a positive integer or $\infty$ and let $\frakA_{v}$ be the subset of $\AW$ (including $\bf1$ \footnote{In order to consider $x$ to be an associate in Section \mref{sec:generating} , we add the trivial averaging word $\bf1$ to $\AW$.}) where the length of $x$-runs is $ \leq v$ with the convention that there is no restriction on $x$-runs when $v=\infty$. The number of balanced pairs of brackets (resp. of $x$) in an averaging word is called its {\bf degree} (resp. {\bf arity}). For $n \geq 1$, let $\frakA_{v}(n)$ denote the subset of $\frakA_{v}$ consisting of all averaging words of degree $n$. For $m \geq 1$, let $\frakA_{v}(n,m)$ denote the subset of $\frakA_{v}$ consisting of averaging words with degree $n$ and arity $m$. Moreover, for $1 \leq k \leq m$, we let $\frakA_{v}(n,m;k)$ be the subset of $\frakA_{v}(n,m)$ consisting of averaging words where the $m$ $x$'s are distributed into exactly $k$ $x$-runs.

Let $G(m,k,v)$ be the set of compositions of the integer $m$ into $k$ positive integer parts, with each part at most $v$ and let $g(m, k , v)$ be the size of this set ($v =\infty$ means there is no restrictions on the size of each part).   Then we have~\mcite{GS,MPA}
\begin{equation}\label{eq:finite}
G_{k,v}(t) := \sum_{m=1}^{\infty}g(m,k,v) t^{m} = t^{k} \left( \frac{1-t^{v}}{1-t} \right)^{k}
\end{equation}
and
\begin{equation}\label{eq:infinite}
G_{k,\infty}(t) :=\sum_{m=1}^{\infty}g(m,k,\infty) t^{m} = \left(\frac{t}{1-t} \right)^{k}.
\end{equation}
In particular, by Eq. (\mref{eq:finite}), we have $v = 1$ implies $G_{k,1}(t) = t^{k}$.

By the definition of $\frakA_{v} (n,m)$, we have the disjoint union
\begin{equation}\mlabel{eq:disjoint}
\frakA_{v}(n,m) = \bigsqcup_{k=1}^{m} \frakA_{v}(n,m;k).
\end{equation}
By the definition of $\frakA_{1}(n,m;k)$, we have that $\frakA_{1}(n,m;k) \neq \emptyset$ implies $m = k$. Then we have $\frakA_{1}(n,m;k) = \frakA_{1}(n,k)$. We define a map
\begin{equation}
\Phi_{v,n,m} : \frakA_{v}(n,m;k) \longrightarrow \frakA_{1}(n,k)
\end{equation}
by sending $w \in \frakA_{v}(n,m;k)$ to the averaging word $\Phi_{v,n,m}(w)$ in $\frakA_{1}(n,k)$ obtained by replacing each of the $x$-runs appearing in $w$ by a single $x$. This map is clearly surjective for each pair $(n,k)$. Further, each fiber (inverse image) of $\Phi_{v,n,m}$ has $g(m,k,n)$ elements, giving rise to a bijection
\begin{equation}
\Psi: \frakA_{v}(n,m;k) \longleftrightarrow  \frakA_{1}(n,k) \times G(m, k, v).
\end{equation}

Therefore by Eq.~(\mref{eq:disjoint}), the numbers $\fraka_{v}(n,m)$ of averaging words of degree $n$ and arity $m$ in the set $\frakA_{v}$ are given by
$$
\fraka_v(n,m)=\sum_{k=1}^{m} g(m,k,v) \fraka_{1}(n,k).
$$

We next determine the expressions of the generating functions $\frakA_{v}(z,t)$ of the number sequences $\fraka_{v}(n,m), n, m\geq 0,$ for $1 \leq v \leq \infty$.

\begin{theorem}
Let $1 \leq v \leq \infty$. The generating function $\frakA_{v}(z,t)$ for the
number $\fraka_{v}(n,m), n, m\geq 0,$ of averaging words is given by  $\frakA_{v}(z,t) = \frakA_{1} (z,G_{1,v}(t))$.
where $G_{1,v}$ is given by Eq. (\ref{eq:finite}) for finite $v$ and by Eq. (\ref{eq:infinite})
for infinite $v$.
\end{theorem}

\begin{proof} We have
\begin{eqnarray*}
\frakA_{v}(z,t): &=& \sum_{n,m \geq 1} \fraka_{v}(n,m)z^{n}t^{m} \\
&=&\sum_{n,m \geq 1} \sum_{k=1}^{m} g(m,k,v) \fraka_{1}(n,k) z^{n}t^{m}  \\
&=& \sum_{m \geq 1} \left( \sum_{k \geq 1} g(m,k,v) \left(\sum_{n \geq 1}\fraka_{1}(n,k) z^{n} \right) \right)t^{m}\\
&=& \sum_{k \geq 1}  \left(\sum_{n \geq 1}\fraka_{1}(n,k) z^{n} \right) (G_{1,v}(t))^{k}\\
&=& \frakA_{1}(z,G_{1,v}(t)).
\end{eqnarray*}
\end{proof}

By Eq.~(\mref{eq:infinite}) we also have
\begin{coro}
We have the generating function
$$
\frakA_{\infty}(z,t) = \frakA_{1} \left(z, \frac{t}{1-t} \right).
$$
\end{coro}

Now we have reduced the problem of finding the explicit expression of the generating function for $\fraka_{\infty}(n,m)$ to the problem of finding $\frakA_{1}(z,t)$, to be considered in the next subsection.

\subsection{The generating function $\frakA_{1}(z,t)$ }\mlabel{sec:generating}
In this section, we will focus on the generating function $\frakA_{1}(z,t)$ for $\fraka_{1}(n,m)$. First, we give some descriptions of the word structures of averaging words. Note that $\frakA_{1}$ is the subset consisting of ${\bf 1}$ and averaging words $w$ composed of $x$'s and pairs of balanced brackets such that no two $x$'s are adjacent, and no two pairs of brackets can be immediately adjacent or nested. This special case can be considered as the case where we assume $X = \{x\}$, $x^{2} = x$ and $\lfloor\ \rfloor^{2} = \lfloor\ \rfloor$.

{\it In the rest of this section, all averaging words are assumed to be in $\frakA_{1}$.}

For $n >0$, let $B(n)$ be the subset of $\frakA_{1}(n)$ consisting of averaging words that begin with a left bracket and end with a right bracket and words in $B(n)$ are said to be {\bf bracketed}. By pre- or post-concatenating a bracketed averaging word $w$ with $x$, we
get three new averaging words: $xw$, $wx$, and $xwx$, which are called respectively the {\bf left,
right and bilateral associate} of $w$. Moreover, left,
right and bilateral associate of $w$ are collectively referred to as the {\bf associates} of $w$. We also consider $x$ to be an associate
of the trivial word ${\bf 1}$. Any nontrivial averaging word is either bracketed or an associate.
Thus for $n > 0$, the set $C(n)$ of all associates is the disjoint union
$$
C(n) = xB(n) \sqcup B(n)x \sqcup xB(n)x, \quad n > 0.
$$
and forms the complement of $B(n)$ in $\frakA_{1}(n)$. Define $C^{-}(n) := B(n)x$. We call the words in the subset $xB(n) \sqcup xB(n)x$ {\bf admissible}. The set of bracketed averaging words is further divided into two disjoint subsets. The first
subset $I(n)$ consists of all indecomposable bracketed averaging words whose beginning
left bracket and ending right bracket are paired, like $\lc x\lc x\rc\rc$. The second subset $D(n)$
consists of all decomposable bracketed averaging words whose beginning left bracket
and ending right bracket are not paired, like $\lc x\rc x\lc x\rc$. For the convenience in counting, we define
$B(0), I(0), D(0)$ to be the empty set and note that $C(0)=\{x\}$. With our convention, denote $\frakA(0) := \frakA_{1}(0) = \{{\bf 1},x\}$.

The following table lists these various types of averaging words in lower degrees
\begin{table}[hbtp] \begin{center}\begin{tabular}{|c|c|c|c|c|c|} \hline $\rm deg$ & $\rm I(n)$ & $\rm D(n)$ & $\rm C(n)$& $ \rm C^{-}(n)$&$ \rm B(n)$ \\ \hline
$0$& & & $x$ & & \\ \hline
$1$& $\lfloor x \rfloor$& & $x \lfloor x \rfloor, \lfloor x \rfloor x , x \lfloor x \rfloor x $&  $\lfloor x \rfloor x$ & $\lfloor x \rfloor $ \\ \hline
$2$ & $ \lfloor x \lfloor x \rfloor \rfloor, \lfloor x \lfloor x \rfloor x \rfloor$ &$\lfloor x \rfloor x \lfloor x \rfloor$ & $9$ associates & ${\rm B(2)}x$ & $\rm I(2) \cup D(2)$\\ \hline
\end{tabular}
\end{center}\label{exampletable}
\end{table}

The production rules will be
\begin{eqnarray}
\langle AW \rangle &\longrightarrow& {\bf 1}~ |~ \langle bracketed \rangle ~ |~ \langle associate \rangle \mlabel{rule1}\\
\langle associate \rangle &\longrightarrow & x ~|~ x \langle bracketed \rangle~|~ \langle bracketed \rangle x ~|~ x \langle bracketed \rangle x \mlabel{rule2}\\
\langle bracketed \rangle &\longrightarrow & ~  \langle indecomposable \rangle ~|~ \langle decomposable \rangle \mlabel{rule3}\\
\langle indecomposable \rangle &\longrightarrow& ~ \lfloor \langle admissable \rangle \rfloor \mlabel{rule4}\\
\langle  decomposable \rangle &\longrightarrow & ~ \langle bracketed \rangle x \langle bracketed \rangle \mlabel{rule5}
\end{eqnarray}

An averaging word $w$ has
arity $m$ means the number of occurrences of $x$ in $w$ is $m$. For any $m \geq 0$, let $\frakA(n,m)$ be the subset of $\frakA$ consisting of words with degree $n$ and arity
$m$, and define similarly the notations $C(n,m)$, $C^{-}(n,m)$, $B(n,m)$, $I(n,m)$, and $D(n,m)$.
These are all finite sets. Let their sizes be $\fraka_{n,m}$, $c_{n,m}$, $c^{-}_{n,m}$,
$b_{n,m}$, $i_{n,m}$, and $d_{n,m}$ respectively. For initial values, we have
\begin{eqnarray*}
\fraka_{0,0} = 1; && c_{0,0} = b_{0,0} = i_{0,0} = d_{0,0} = 0;\\
\fraka_{0,1} = c_{0,1} = 1; && b_{0,1} = i_{0,1} = d_{0,1} = 0;\\
\fraka_{1,1} = b_{1,1} = i_{1,1} = 1;  && c_{1,1} = d_{1,1} = 0;\\
\fraka_{1,2} = c_{1,2} = 2; && b_{1,2} = i_{1,2} = d_{1,2} = 0;\\
\fraka_{1,3} = c_{1,3} = 1; && b_{1,3} = i_{1,3} = d_{1,3} = 0;
\end{eqnarray*}
\begin{eqnarray*}
\fraka_{0,m} = c_{0,m} = b_{0,m} = i_{0,m} = d_{0,m} = 0 && \mbox{for}~ m \geq 2;\\
\fraka_{1,m} = c_{1,m} = b_{1,m} = i_{1,m} = d_{1,m} = 0 && \mbox{for}~ m \geq 4;\\
\fraka_{n,0} = c_{n,0} = b_{n,0} = i_{n,0} = d_{n,0} = 0 && \mbox{for}~ n \geq 1;\\
\fraka_{n,1} = c_{n,1} = b_{n,1} = i_{n,1} = d_{n,1} = 0 && \mbox{for}~ n \geq 2.
\end{eqnarray*}

From the production rules (\ref{rule1})--(\ref{rule5}), we see that for $n \geq 1, m \geq 2$,
\begin{eqnarray}
\fraka_{n,m} &=& b_{n,m} + c_{n,m}, \mlabel{eq:formula1}\\
c_{n,m} &=& 2 b_{n,m-1} + b_{n,m-2}, \mlabel{eq:formula2}\\
b_{n,m} &=& i_{n,m} + d_{n,m}, \mlabel{eq:formula3}\\
i_{n,m} &=& c_{n-1,m} - c^{-}_{n-1,m} = c_{n-1,m} - b_{n-1,m-1}, \mlabel{eq:formula4}
\end{eqnarray}
where Eq.~(\ref{eq:formula4}) follows from Eq.~(\ref{rule4}) and $C^{-}(n) = B(n)x$.

Now for $n \geq 2$, $m \geq 2$ and $w \in D(n,m)$, we can write $w$ uniquely as
$w_{n_{1}}xw_{n_{2}} \cdots x w_{n_{p}}$ where $w_{n_{j}} \in I(n_{j})$ and $n_{1}+\cdots +n_{p}$ is a composition of $n$ using
$p$ positive integers. Let $m_{j}$ be the arity of $w_{n_{j}}$. Then clearly, $m_{1} + \cdots +m_{p} =
m - p + 1$. So we have
\begin{equation}
d_{n,m} =\sum^{min(n,m)}_{p=2} \sum_{(m_{1},\cdots,m_{p};m-p+1)} \sum_{(n_{1},\cdots,n_{p};n)}(i_{n_{1},m_{1}}) \cdots (i_{n_{p},m_{p}}).
\end{equation}
The case when $p = 1$ corresponds to a single summand $i_{n,m}$, and then
\begin{equation} \mlabel{bracket2}
b_{n,m} =i_{n,m} + d_{n,m} =
\sum^{min(n,m)}_{p=1} \sum_{(m_{1},\cdots,m_{p};m-p+1)} \sum_{(n_{1},\cdots,n_{p};n)}(i_{n_{1},m_{1}}) \cdots (i_{n_{p},m_{p}}).
\end{equation}
Now from Eqs.(\ref{eq:formula2}) and (\ref{eq:formula4}), we have
\begin{equation}
i_{n,m} = c_{n-1,m} - b_{n-1,m-1} = 2 b_{n-1,m-1} + b_{n-1,m-2} - b_{n-1,m-1}= b_{n-1,m-1} + b_{n-1,m-2}.
\mlabel{relation2}
\end{equation}
Define the bivariate generating series
$$
\frakA(z, t) = \sum_{n=0}^{\infty} \sum_{m=0}^{\infty} \fraka_{n,m}z^{n} t^{m}
$$
and similarly define $B(z, t)$, $I(z, t)$, $D(z, t)$, and $C(z, t)$. Note that for $B(z, t)$,
$I(z, t)$ and $D(z, t)$, it does not matter whether the series indices $n,m $ start at $0$
or $1$. We will multiply both sides of Eq.(\ref{relation2}) by $z^{n}t^{m}$ and sum up for $n \geq 2$, $m \geq 2$. Since $i_{0,m}=i_{n,0}=0$ for all $m,n \geq 0$ and $i_{1,m} = i_{n,1}=0$ for all $m,n \geq 2$, then the left hand side gives
\begin{eqnarray*}
\sum_{n=2}^{\infty}\sum_{m=2}^{\infty}i_{n,m}z^{n}t^{m} = I(z, t) - z t.
\end{eqnarray*}
Now, we sum up the right hand side of Eq.(\ref{relation2}) one term at a time.
$$
\sum_{n=2}^{\infty} \sum_{m=2}^{\infty}
b_{n-1,m-1} z^{n}t^{m} = zt\sum_{n=1}^{\infty}\sum_{m=1}^{\infty}b_{n,m}z^{n}t^{m} =z t B(z,t),
$$
$$
\sum_{n=2}^{\infty}\sum_{m=2}^{\infty} b_{n-1,m-2} z^{n}t^{m} = z t^{2}\sum_{n=1}^{\infty}\sum_{m=0}^{\infty}b_{n,m} z^{n}t^{m}
= zt^{2} ( B(z, t) + \sum_{n=1}^{\infty}b_{n,0}z^{n}) =zt^{2} B(z, t).
$$
Hence, we have the identity
\begin{equation}
I(z, t) - zt = zt(1 + t) B(z, t).
\end{equation}

Using Eq.~(\ref{bracket2}), we have
\begin{eqnarray*}
\lefteqn{\sum_{n=1}^{\infty}\sum_{m=1}^{\infty} b_{n,m}z^{n}t^{m}= \sum_{n=1}^{\infty}\sum_{m=1}^{\infty} \sum_{p=1}^{\min (m,n)}  \sum_{(m_{1},\cdots,m_{p}; m-p+1)} \sum_{(n_{1},\cdots,n_{p};n)}
(i_{n_{1},m_{1}}) \cdots (i_{n_{p},m_{p}}) z^{n}t^{m}} \\
&=& \sum_{n=1}^{\infty}\sum_{m=1}^{\infty} \sum_{p=1}^{\min (m,n)}  \sum_{(m_{1},\cdots,m_{p}; m-p+1)} \sum_{(n_{1},\cdots,n_{p};n)}
(i_{n_{1},m_{1}} z^{n_{1}} t^{m_{1}}) \cdots (i_{n_{p},m_{p}} z^{n_{p}} t^{m_{p}}) z^{n}t^{m} \\
&=& \sum_{p=1}^{\infty} \left(  \sum_{k=1}^{\infty} \sum_{\ell =1}^{\infty}  i_{k,\ell} z^{k} t^{\ell} \right)^{p} t^{p-1}\\
&=& \sum_{p=1}^\infty I(z,t)^p t^{p-1}
\end{eqnarray*}
and hence
\begin{equation}
B(z,t) = \frac{I(z,t)}{ 1- tI(z,t)}\,.
\mlabel{eq:brack}
\end{equation}
Thus we obtained the identity defining $I(z,t)$ as
\begin{equation}
I(z, t) - zt = zt(1 + t)  \frac{I(z,t)}{ 1- tI(z,t)}.
\end{equation}
Solving this quadratic equation in $I(z,t)$ and using the initial conditions, we find
\begin{equation}
I(z,t) = \displaystyle \frac{1-zt - \sqrt{z^{2}t^{2} - (2t + 4t^{2}) z +1 }}{2t}\,.
\mlabel{eq:ind}
\end{equation}
Then by Eq.~(\mref{eq:brack}), we have
\begin{equation}
B(z,t) = \displaystyle \frac{1-zt -2zt^{2} - \sqrt{z^{2}t^{2} - (2t + 4t^{2}) z +1 }}{2zt^{2}(1+t)}.
\mlabel{eq:bracket}
\end{equation}
Furthermore, from Eq.~(\mref{bracket2}), we obtain
\begin{eqnarray}
D(z,t)
&=& B(z,t)-I(z,t)\mlabel{eq:dec}\\
&=&\displaystyle \frac{1-2zt -3zt^{2} + z^{2} t^{2} + z^{2}t^{3} + (zt + zt^{2} -1) \sqrt{z^{2}t^{2} - (2t + 4t^{2}) z +1 } } {2zt^{2}(1+t)}. \notag
\end{eqnarray}

We can also obtain the bivariate generating series for $c_{n,m}$:
\begin{eqnarray*}
C(z,t) &=&\sum_{n=0}^{\infty}(c_{n,0}z^{n} + c_{n,1}z^{n}t + \sum_{m=2}^{\infty}c_{n,m}z^{n}t^{m})\\
&=&  t + \sum_{m=2}^{\infty}c_{0,m}t^{m} +\sum_{n=1}^{\infty}\left( \sum_{m=2}^{\infty}c_{n,m}z^{n}t^{m}\right)\\
&=& t +\sum_{n=1}^{\infty}\sum_{m=2}^{\infty} (2 b_{n,m-1} + b_{n,m-2})z^{n}t^{m}\\
&=& t +\sum_{n=1}^{\infty}\left( t \sum_{m=1}^{\infty}2 b_{n,m}z^{n}t^{m} \right)+ t^{2}\sum_{n=1}^{\infty}\sum_{m=0}^{\infty}b_{n,m}z^{n}t^{m}\\
&=& t + 2tB(z, t) + t^{2} B(z, t),
\end{eqnarray*}
giving
\begin{equation}\mlabel{eq:associative}
A(z,t) = \frac{2+t - 2zt - 3zt^{2} - (2+t) \sqrt{z^{2}t^{2} - (2t+4t^{2}) z +1 }} {2zt(1+t)}
\end{equation}
by Eq.~(\mref{eq:bracket}).

Finally, utilizing $\frakA(z,t) = 1+B(z,t)+C(z,t)$,  Eqs.(\ref{eq:bracket}) and (\ref{eq:associative}) we derive the following generating function $\frakA_{1} (z,t) = \frakA(z,t)$ for the sequences $\fraka_{n,m} = \fraka_{1}(n,m)$.
\begin{equation}\mlabel{eq:double}
\frakA(z,t) = \frac{(1+t) (1 -zt - \sqrt{z^{2}t^{2} - (2t+4t^{2}) z +1 })} {2zt^{2}}.
\end{equation}

Recall that the averaging words in $\mathfrak{A}_{1}$ can be considered as the case where we take $X = \{x\}$, $x^{2} = x$ and $\lfloor\ \rfloor^{2} = \lfloor\ \rfloor$.
We conclude our discussion on generating functions by ignoring the arity and only focus on the degree. For $n > 0$, let $\fraka_{n}$ (resp. $c_{n}$, resp. $b_{n}$, resp. $i_{n}$, resp. $d_{n}$) be the number of all
(resp. associate, resp. bracketed, resp. indecomposable, resp. decomposable)
averaging words with $n$ pairs of (balanced) brackets.
For example, we have $\fraka_0=2$ since such averaging words with no operators are $\{\bfone, x\}$; while $\fraka_1=4$ since  those with one operators are $\{\,\lc x\rc, x\lc x\rc, \lc x\rc x, x\lc x\rc x\,\}.$

Putting $t=1$ in the generating functions of $\frakA(z,t), B(z,t), I(z,t), D(z,t)$ and $C(z,t)$, we obtain

\begin{theorem}\mlabel{thm:generation1}
The generating series for $\fraka_{n}$, $b_{n}$, $i_{n}$, $d_{n}$ and $c_{n}$ are given by
\begin{eqnarray}
\frakA(z) & = & \sum_{n=0}^{\infty} \fraka_{n} z^{n} =  \displaystyle \frac{1 - z - \sqrt{z^{2} - 6z +1}}{z},\notag\\
B(z) &=& \sum_{n=0}^{\infty} b_{n} z^{n} = \displaystyle  \frac{1-3z- \sqrt{z^{2} - 6z +1}}{4z}, \notag\\
I(z) &=& \sum_{n=0}^{\infty} i_{n} z^{n} = \displaystyle \frac{1-z - \sqrt{z^{2} -6z+1}}{2}, \mlabel{eq:generation1}\\
D(z) &=& \sum_{n=0}^{\infty} d_{n} z^{n} = \displaystyle \frac{1-5z +2z^{2} +(2z-1) \sqrt{z^{2} - 6z +1}}{4z},\notag \\
C(z) &=& \sum_{n=0}^{\infty} c_{n} z^{n} = \displaystyle \frac{3-5z - 3 \sqrt{z^{2} - 6z +1}}{4z}.\notag
\end{eqnarray}
\end{theorem}

From the theorem we obtain
\begin{coro}
\begin{enumerate}
\item
The sequence $\fraka_n, n\geq 0$, of averaging words of degree $n$ with one idempotent generator and one idempotent operator, is twice the sequence $s_n$ of large Schr\"oder numbers:
$\fraka_n/2=s_n, n\geq 0$. The first few terms of $\fraka_n, n\geq 0,$ are
$$
2,4,12,44,180,788,3612,17116,\cdots.
$$
\item
The sequence $i_{n}, n\geq 1,$ of bracketed indecomposable averaging words of degree $n$ is the sequence $s_n, n\geq 0,$ of large  Schr\"oder numbers: $i_{n+1}=s_n, n\geq 0$. The first few terms of $i_n, n\geq 0,$ are
$$
0,1,2,6,22,90,394,1806,\cdots .
$$
\end{enumerate}
\mlabel{co:seq}
\end{coro}

\begin{proof}
Both results are proved by comparing the corresponding generating functions with the generating function $$S(z):=\sum_{n=0}^\infty s_nz^n=\displaystyle \frac{1-z - \sqrt{z^{2} -6z+1}}{2z}$$
of large Schr\"oder sequence (A006318 in the On-line Encyclopedia of Integer Sequences~\mcite{SN}).
\end{proof}

On the other hand, the number $d_{n}$ of bracketed decomposable averaging words of degree $n$ ($n \geq 0$) is a sequence which starts with
$$
0,0,1,5, 23, 107, 509,2473,\cdots.
$$
This sequences is new and can not be found in~\mcite{SN}.

\subsection{Averaging words and large Schr\"oder numbers}
The sequence $s_n, n\geq 0,$ of large Schr\"oder numbers (A006318 in \cite{SN}) is an important sequence of integers with numerous interesting properties and interpretations. For example, $s_n$ counts the number of Schr\"oder paths of semilength $n$, namely lattice paths on the plane from $(0, 0)$ to
$(2n, 0)$ that do not go below the $x$-axis and consist of up steps $U = (1, 1)$,
down steps $D = (1,-1)$ and horizontal steps $H = (2, 0)$. See~\mcite{SN} for more details.

Corollary~\mref{co:seq} gives two more interpretations of large Schr\"oder numbers in terms of averaging words with a single idempotent generators and idempotent operator.
Motivated by this, we next give another interpretation of the sequence of large Schr\"oder numbers in terms of decorated trees.

\begin{defn}
{\rm
\begin{enumerate}
\item
Let $\mathfrak{T}$ denote the set of planar reduced rooted trees together with the trivial tree $\idx$. Let $\omega$ and $\iota$ be symbols. By a {\bf $(\omega,\iota)$-decorated tree} we mean a tree $t$ in $\mathfrak{T}$ together with a decoration on the vertices of $t$ by $\omega$ and a decoration on the leaves of $t$ by either $\omega$ or $\iota$.
Let $\mathcal{D}(t)$ denote the set of $(\omega,\iota)$-decorated trees from $t$ and denote
\begin{equation*}
\mathcal{D}(\mathfrak{T}):=\coprod_{ t \in \mathfrak{T}} \mathcal{D}(t).
\end{equation*}
\item
Let $t$ be in $\mathfrak{T}$ whose number of leaves is greater than $1$. Then there exists an integer $m$ such that $t$ can be written uniquely as the {\bf grafting} $\bigvee_{i=1}^{m} t_{i}:=t_1\vee \cdots \vee t_m$ of trees $t_{1},t_{2}, \cdots t_{m}$. The trees $t_1,\cdots,t_m$ are called the {\bf branches} of $t$.
\item
Let $\tau\in \mathcal{D}(t)$ where $t$ is in $\mathfrak{T}$ with grafting $t=t_1\vee \cdots \vee t_m$. Then there are $\tau_i\in \mathcal{D}(t_i), 1\leq i\leq m$, such that $\tau$ is the grafting $\omega(\bigvee_{i=1}^m \tau_i)=\omega(\tau_1\vee \cdots\vee\tau_m)$ of $\tau_1,\cdots,\tau_m$
with the new root decorated by $\omega$.
\end{enumerate}}
\end{defn}
Since the rooted trees we are considering are reduced, we have $m\geq 2$ in any grafting of trees.

Now, we define a special subset of $\mathcal{D}(\mathfrak{T})$.

\begin{defn}
\begin{enumerate}
\item
A $(\omega,\iota)$-decorated tree $\tau\in \mathcal{D}(\mathfrak{T})$ is called a {\bf Schr\"oder tree} if either $\tau$ is the trivial tree decorated by $\omega$, or $\tau$ is a tree with more than one leaves and satisfies the following conditions: For each vertex $v$ of $\tau$, let $\tau_v$ be the subtree of $\tau$ with root $v$. Then
\begin{enumerate}
\item the leftmost branch of $\tau_v$ is a leaf decorated by $\iota$;
\item the branches of $\tau_v$ are alternatively a leaf decorated by $\iota$ and a subtree that is not a leaf decorated by $\iota$ (the latter means that the subtree is either not a leaf or a leaf decorated by $\omega$).  \end{enumerate}
To put it in another way, let $\tau_v$ be of the form $\omega(\tau_{v,1}\vee \cdots \tau_{v,k})$, then each $\tau_{v,i}$ for $i$ odd is a leaf decorated by $\iota$ and
each $\tau_{v,i}$ for $i$ even is either not a leaf or is a leaf decorated by $\omega$.
\item
Let $\Sh$ denote the set of Schr\"oder trees and let $\Sh_n$ denote the subset of $\Sh$ consisting of those Schr\"oder trees whose n vertices or leaves are decorated by~$\omega$.
\end{enumerate}
\end{defn}

\begin{theorem}
\begin{enumerate}
\item
The sequence $\{|\Sh_{n}|\}_{n \geq 1}$ counting Schr\"oder trees is the sequence $\{i_n\}_{n\geq 1}$ counting indecomposable averaging words: $|\Sh_n|=i_n, n\geq 1$.
\mlabel{it:count1}
\item
The sequence $\{|\Sh_{n}|\}_{n \geq 1}$ coincides the sequence $\{s_n\}_{n\geq 0}$ of large Schr\"oder numbers. In other words,
$$ |\Sh_{n+1}|=i_{n+1}=s_n, n\geq 0.
$$
\mlabel{it:count2}
\item
The sequence $\{s_n\}_{n\geq 0}$ of large Schr\"oder numbers satisfies the following recursion:
\begin{eqnarray}
s_0&=& 1,\notag\\
s_n&=&2\sum_{j=1}^{n} \sum_{(p_{1},\cdots,p_{j})\in G(n,j)} s_{p_{1}-1}\cdots s_{p_{j}-1},
\mlabel{eq:indw}
\end{eqnarray}
where $G(n,j)$ is the set of compositions of the integer $n$ into $j$ positive integer parts.
\mlabel{it:count3}
\end{enumerate}
\mlabel{thm:count}
\end{theorem}

We illustrate the theorem by listing the first three terms of the sequences $\mathrm{I}(n)$ and $\Sh_{n}, n\geq 1$.
\begin{enumerate}
\item[{\bf 1.}]
\begin{table}[hbtp]\begin{center}\begin{tabular}{|c|c|c|}
\hline $\mathrm{I}(1)$ & $\lc  x \rc$&1\\ \hline
$\mathrm{I}(2)$ & $ \lc x \lc x \rc \rc, \lc x \lc x \rc x \rc$& 2 \\ \hline
$\mathrm{I}(3)$&$\lc x  \lc x \lc x \rc \rc \rc, \lc x  \lc x \lc x \rc \rc x \rc, \lc x \lc x \lc x \rc x \rc \rc, \lc x \lc x \lc x \rc x \rc x\rc, \lc x \lc x \rc x \lc x \rc \rc, \lc x \lc x \rc x \lc x \rc x\rc$& 6 \\ \hline
\end{tabular}
\end{center}\label{exampleword}
\end{table}

\item[{\bf 2.}]
$\Sh_{1}:$  $\begin{picture}(6,6)(-3,-3)
\put(0,0){\line(0,1){6}}
\put(0,0){\line(0,-1){6}}
\put(-2.5,7){$\omega$}
\end{picture}$;
\vspace{0.5cm}

\noindent
$\Sh_{2}:$ \qquad  $\begin{picture}(5,5)(0,0)
\line(-1,1){10}
\put(-10,10){$\iota$}
\put(10,10){$\omega$}
\put(-2.5,-2.5){$\bullet$}
\line(0,-1){10}\line(1,1){10}
\end{picture}$ \quad,  \qquad
$\begin{picture}(10,10)(0,0)
\put(-5,10){\line(0,-1){23}$\omega$}
\put(-8,-5){$\bullet$}
\put(-6,-4){\line(1,1){13}}
\put(-5,-4){\line(-1,1){13}}
\put(7,10){$\iota$}
\put(-18,10){$\iota$}
\end{picture}$;
\vspace{0.7cm}

\noindent
$\Sh_{3}:$ \qquad $\begin{picture}(10,10)(0,0)
\put(-5,0){\line(0,-1){15}}
\put(-8,-5){$\bullet$}
\put(-6,-4){\line(1,1){20}}
\put(1,2){$\bullet$}
\put(-5,-4){\line(-1,1){13}}
\put(-18,10){$\iota$}
\put(4,4){\line(-1,1){10}}
\put(15,15){$\omega$}
\put(-5,13){$\iota$}
\end{picture}$ , \quad \qquad $\begin{picture}(10,10)(0,0)
\put(-5,5){\line(0,-1){20}}
\put(-8,2){$\bullet$}
\put(-8,-7){$\bullet$}
\put(-6,-5){\line(-1,1){15}}
\put(-22,10){$\iota$}
\put(-5,-6){\line(1,1){15}}
\put(10,10){$\iota$}
\put(-6,6){\line(-1,1){10}}
\put(-5,6){\line(1,1){10}}
\put(5,16){$\omega$}
\put(-18,16){$\iota$}
\end{picture}$ , \quad \qquad $\begin{picture}(10,10)(0,0)
\put(-5,0){\line(0,-1){15}}
\put(-8,-5){$\bullet$}
\put(-6,-4){\line(1,1){17}}
\put(1,2){$\bullet$}
\put(-5,-4){\line(-1,1){13}}
\put(-18,10){$\iota$}
\put(4,4){\line(-1,1){10}}
\put(13,13){$\iota$}
\put(-5,13){$\iota$}
\put(4,4){\line(0,1){10}}
\put(1,15){$\omega$}
\end{picture}$, \quad \qquad $\begin{picture}(10,10)(0,0)
\put(-5,20){\line(0,-1){35}}
\put(-8,2){$\bullet$}
\put(-8,-7){$\bullet$}
\put(-6,-5){\line(-1,1){15}}
\put(-22,10){$\iota$}
\put(-5,-6){\line(1,1){15}}
\put(10,10){$\iota$}
\put(-6,6){\line(-1,1){10}}
\put(-5,6){\line(1,1){10}}
\put(5,16){$\iota$}
\put(-18,16){$\iota$}
\put(-7,20){$\omega$}
\end{picture}$, \quad \qquad $\begin{picture}(10,10)(0,0)
\put(-5,0){\line(0,-1){13}}
\put(-8,-2){$\bullet$}
\put(-5,0){\line(-2,1){18}}
\put(-5,0){\line(2,1){18}}
\put(-5,0){\line(-1,2){7}}
\put(-5,0){\line(1,2){7}}
\put(-24,9){$\iota$}
\put(-15,14){$\omega$}
\put(2,13){$\iota$}
\put(13,9){$\omega$}
\end{picture}$, \quad \qquad $\begin{picture}(10,10)(0,0)
\put(-5,15){\line(0,-1){30}}
\put(-8,-2){$\bullet$}
\put(-5,0){\line(-2,1){18}}
\put(-5,0){\line(2,1){18}}
\put(-5,0){\line(-1,2){7}}
\put(-5,0){\line(1,2){7}}
\put(-24,9){$\iota$}
\put(-19,14){$\omega$}
\put(2,13){$\omega$}
\put(14,9){$\iota$}
\put(-5,16){$\iota$}
\end{picture}$ .
\end{enumerate}
\bigskip

\begin{proof}
(\mref{it:count1}) We just need to prove that the sequence $\{|\Sh_n|\}_{n\geq 1}$ and  $\{i_n\}_{n\geq 1}$ satisfy the same recursion relation and the same initial condition.

Let $\mathrm{I}(n)$ denote the set of bracketed indecomposable averaging words of
degree $n$ and $\mathrm{I}$ by the set of all bracketed indecomposable averaging words.
\medskip

First we have $\mathrm{I}(1) = \{ \lc x \rc  \}$ and $\Sh_{1} = \{\begin{picture}(6,6)(-3,-3)
\put(0,0){\line(0,1){6}}
\put(0,0){\line(0,-1){6}}
\put(-2.5,7){$\omega$}
\end{picture} \}$. Hence $|\Sh_1|=i_1=1$.

Next, let $n\geq 2$. For any word $W = \lc w_{1}w_{2}\cdots w_{m}\rc \in \mathrm{I}(n)$, we have $w_{i}=x$ for $i$ odd and $w_{i}\in \mathrm{I}$ for $i$ even. Then there exist $p_{i}, 1\leq i\leq k:=\lc m/2\rc$ such that $w_{2i}$ is in $\mathrm{I}(p_{i})$. Thus we have $p_{1}+\cdots+p_{k} = n-1$ and so $(p_{1},\cdots,p_{k})\in G(n-1,k)$. Note that there are two $W$'s that give the same $(p_1,\cdots,p_k)$: one is $\lc w_1\cdots w_{2k}\rc$, the another is $\lc w_1\cdots w_{2k} x\rc$.

Conversely, let $(p_{1}, \cdots, p_{k}) \in G(n-1,k)$ and take $w_{2} \in I(p_{1}), \cdots, w_{2k} \in I(p_{k})$ and $w_{1} = \cdots = w_{2k-1} = w_{2k+1}=x$. Then $\lc w_{1} \cdots w_{2k} \rc$ and $\lc w_{1}\cdots w_{2k+1}\rc$ are in $\mathrm{I}(n)$.

Therefore, we have the following recursive formula for $i_{n}$, $n \geq 2$:
\begin{equation}
i_{n} = 2\left(i_{n-1}+\cdots + \sum_{(p_{1},\cdots,p_{j})\in G(n-1,j)} i_{p_{1}}\cdots i_{p_{j}}+ \cdots + i_{1}^{n-1}\right)=2\sum_{j=0}^{n-1} \sum_{(p_{1},\cdots,p_{j})\in G(n-1,j)} i_{p_{1}}\cdots i_{p_{j}}.
\mlabel{eq:indi}
\end{equation}

On the other hand, if $\tau \in \Sh_{n}$, then there exists $k\geq 0$ such that either $\tau= \omega(\bigvee_{i=1}^{2k}\tau_{i})$ or $\omega(\bigvee_{i=1}^{2k+1}\tau_{i})$. Furthermore, each $\tau_{i}$ with $i$ odd is a leaf decorated by $\iota$ and each $\tau_{2i}, 1\leq i\leq k,$ is in $\Sh_{p_{i}}$ for some integer $p_i\geq 1$. Since $\omega$ does not appear in $\tau_{i}$ for $i$ odd, we have $p_{1}+p_{2} + \cdots + p_{k} = n-1$. That is $(p_{1}, \cdots, p_{k})$ is in $G(n-1,k)$.

Conversely, let $(p_{1}, \cdots, p_{k})$ be in $G(n-1,k)$. Take $\tau_{2i} \in \Sh_{p_{i}}, 1\leq i\leq k$ and take $\tau_{i}$ for odd $i$ to be a leaf decorated by $\iota$. Then the $(\omega,\iota)$-decorated trees $\omega(\bigvee_{i=1}^{2k} \tau_{i})$ and $\omega(\bigvee_{i=1}^{2k+1}\tau_{i})$ are in $\Sh_{n}$.

By the above argument, we obtain the following recursive formula for $|\Sh_{n}|$, $n \geq 2$.
\begin{eqnarray}
|\Sh_{n}|& =& 2\left(|\Sh_{n-1}|+\cdots + \sum_{(p_{1},\cdots,p_{j})\in G(n-1,j)} (|\Sh_{p_{1}}|\cdots|\Sh_{p_{j}}|)+ \cdots + |\Sh_{1}|^{n-1}\right)\mlabel{eq:indt}\\
&=&\sum_{j=0}^{n-1}\sum_{(p_{1},\cdots,p_{j})\in G(n-1,j)} (|\Sh_{p_{1}}|\cdots|\Sh_{p_{j}}|).\notag
\end{eqnarray}

In summary, $i_{n}, n\geq 1$ and $|\Sh_{n}|, n\geq 1$ have the same initial value and the same recursive relation. Therefore $i_{n} = |\Sh_{n}|, n\geq 1$.
\smallskip

\noindent
(\mref{it:count2}) This follows from Item~(\mref{it:count1}) and Corollary~\mref{co:seq}.
\smallskip

\noindent
(\mref{it:count3})
This follows from Item~(\mref{it:count2}) and Eq.~(\mref{eq:indt}).
\end{proof}

\begin{remark}
As pointed out by the referee, the recursive formula in Eq.~(\mref{eq:indw}) also follows from the classical identity
$$
zS(z)^2+(z-1)S(z)+1=0
$$
of the generating function of large Schr\"oder sequence, since the identity can be reformulated as
$$ S(z)=1+\frac{2zS(z)}{1-zS(z)},$$
which, when developed, gives Eq.~(\mref{eq:indw}). The proof we provided reveals better the combinatorics behind the recursion. This will be useful for example in the following remark.
\end{remark}

\begin{remark}
Theorem~\ref{thm:count}.(\ref{it:count1}) suggests that, for each $n\geq 1$,  $\Sh_n$ is in bijection with the set of indecomposable averaging words $\mathrm{I}(n)$ of degree $n$, and hence with the set of admissible averaging words of degree $n$ thanks to Eq.~(\mref{rule4}). We show that this is indeed the case by inductively constructing bijections $\Psi_n:\mathrm{I}(n)\to \Sh_n, n\geq 1$, using the proof of Theorem \ref{thm:count}.(\ref{it:count1}).

\smallskip

First, define $\Psi_1(\lc x \rc) = \begin{picture}(6,6)(-3,-3)
\put(0,0){\line(0,1){6}}
\put(0,0){\line(0,-1){6}}
\put(-2.5,7){$\omega$}
\end{picture} $. Assume that $\Psi_k$ has been defined for $ 1 \leq k \leq n$. As in the proof of Theorem~\mref{thm:count}.(\ref{it:count1}), any word $W\in \mathrm{I}(n+1)$ is of the form $W = \lc w_{1}w_{2}\cdots w_{m}\rc$, where $w_{i}=x$ for $i$ odd and $w_{i}\in \mathrm{I}(p_{2/i})$ for $i$ even, with $1 \leq p_{i/2} \leq n$ and $p_1+\cdots+p_{\lc m/2\rc}=n$. Then define
$$
\Psi_{n+1}(W) = \omega\left(\bigvee_{i=1}^m \Psi'(w_{i})\right) \quad \mbox{where} \quad \Psi'(w_{i}) = \left\{\begin{array}{ll} \begin{picture}(6,6)(-3,-3)
\put(0,0){\line(0,1){6}}
\put(0,0){\line(0,-1){6}}
\put(-2.5,7){$\iota$}
\end{picture}& i~\mbox{is odd},\\
\Psi_{p_{2/i}}(w_{i})& i~\mbox{is even}.
\end{array} \right.
$$
For example, we have
$$
\Psi_2(\lc x \lc x \rc \rc)= \quad \begin{picture}(5,5)(0,0)
\line(-1,1){10}
\put(-10,10){$\iota$}
\put(10,10){$\omega$}
\put(-2.5,-2.5){$\bullet$}
\line(0,-1){10}\line(1,1){10}
\end{picture} \quad , \quad
\Psi_2(\lc x \lc x \rc x \rc) = \qquad \begin{picture}(10,10)(0,0)
\put(-5,10){\line(0,-1){23}$\omega$}
\put(-8,-5){$\bullet$}
\put(-6,-4){\line(1,1){13}}
\put(-5,-4){\line(-1,1){13}}
\put(7,10){$\iota$}
\put(-18,10){$\iota$}
\end{picture} \quad, \quad \Psi_3(\lc x \lc x \lc x \rc x \rc x\rc) = \qquad \begin{picture}(10,10)(0,0)
\put(-5,20){\line(0,-1){35}}
\put(-8,2){$\bullet$}
\put(-8,-7){$\bullet$}
\put(-6,-5){\line(-1,1){15}}
\put(-22,10){$\iota$}
\put(-5,-6){\line(1,1){15}}
\put(10,10){$\iota$}
\put(-6,6){\line(-1,1){10}}
\put(-5,6){\line(1,1){10}}
\put(5,16){$\iota$}
\put(-18,16){$\iota$}
\put(-7,20){$\omega$}
\end{picture}.
$$
\end{remark}

\section{Tree representation and operad of averaging algebras}
\label{sec:tree}

In this section, we identify the set $\AW(\{x\})$ of averaging words on a singleton $\{x\}$ with a class of unreduced binary trees. This identification gives a combinatorial description of $\AW(\{x\})$ and hence the free averaging algebra on $\{x\}$, without the idempotency assumption in the previous section. This identification also allows us to construct the operad of averaging algebras.

\subsection{Unreduced binary trees and averaging words on a singleton}
Recall that an unreduced binary tree is a tree in which each vertex has either one input or two inputs. By convention the trivial tree is an unreduced binary tree.
We let $\calt$ denote the set of unreduced binary trees.
A vertex $v$ of $\tau\in \calt$ is called a {\bf uni-vertex} (resp. {\bf bi-vertex}) of $\tau$ if $v$ is a vertex with one input (resp. two inputs). If $v$ is a bi-vertex of $\tau$, then $\tau$ has a left subtree and a right subtree associated to $v$, denoted by $\tau_{v,l}$ and $\tau_{v,r}$ respectively. For any $\tau \in \calt$, denote the number of its leaves by $\lef_{\tau}$.

Denote the binary tree with two leaves by $t_{\mu} := \assop$, the tree with one leaf and one uni-vertex by $t_{P} := \operator$, the trivial tree by $\mathrm{id} := \idx$. We use $t_{\mu} \circ (\tau_{l} \ot \tau_{r})$ to denote the tree obtained by grafting a tree $\tau_{l}$ on the left leaf of $t_{\mu}$ and a tree $\tau_{r}$ on the right leaf of $t_{\mu}$; $t_{P} \circ \tau$ to denote the tree obtained by grafting a tree $\tau$ on the only leaf of $t_{P}$. A {\bf ladder} is a tree of the form $t_{P}^{s}$, $s \geq 1$. For example, we have
$$
t_{\mu} \circ (\mathrm{id} \ot t_{P}) = \begin{picture}(10,10)(-5,-5)
\put(0,0){\line(1,1){12}}
\put(0,0){\line(-1,1){10}}
\line(0,-1){10}
\put(3,3){$\bullet$}
\put(-2.5,-2.5){$\bullet$}
\end{picture}, \quad    t_{P} \circ t_{\mu} = \begin{picture}(10,10)(-5,-5)
\put(0,0){\line(1,1){10}}
\put(0,0){\line(-1,1){10}}
\line(0,-1){13}
\put(-2.5,-2.5){$\bullet$}
\put(-2.5,-10){$\bullet$},
\end{picture}, \quad t_{P}^{2} = \begin{picture}(10,10)(-5,-5)
\line(0,-1){13}
\put(0,0){\line(0,1){8}}
\put(-2.5,-2.5){$\bullet$}
\put(-2.5,-10){$\bullet$},
\end{picture}.
$$

With these notations, any unreduced binary tree $\tau$ with $\lef_{\tau}=1$ (resp. $\lef_{\tau} \geq 2$) can be uniquely expressed as $\tau=t_P^s$ (resp. $\tau=t_P^s\circ t_\mu \circ (\tau_{\ell} \ot \tau_{r})$) with $s\geq 0$ and $\tau_\ell, \tau_r$ unreduced binary trees. We use the convention that $t_{P}^{0} = \id$. The number $\pb(\tau):=s$ is called the {\bf bracketed power} of $\tau$.
An unreduced binary tree $\tau$ is called a {\bf bracketed tree} (resp. {\bf non-bracketed tree}) if $\pb(\tau)>0$ (resp. $\pb(\tau)=0$). For example, $t_{P} \circ t_{\mu}$ and $t_{P}^{2}$ are bracketed trees, while $t_{\mu} \circ (\id \ot t_{P})$ is not. The trivial tree is non-bracketed. An unreduced binary tree with $\lef_{\tau} \geq 2$ is non-bracketed if and only if its root is a bi-vertex.

We now define some special subsets of unreduced binary trees.

An unreduced binary tree $\tau$ is called a {\bf \lft} if
\begin{enumerate}
\item
for each bi-vertex $v$ of $\tau$, its right subtree $\tau_{v,r}$ is a leaf;
\mlabel{it:lfta}
\item
among the tree $\tau$ and the left subtrees of all the bi-vertices of $\tau$, at most one is bracketed.
\mlabel{it:lftb}
\end{enumerate}

For example, in the list of trees
\smallskip
$$
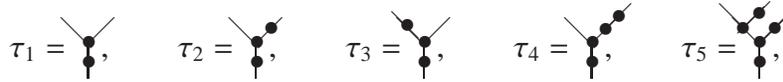
\tau_{1} = \begin{picture}(10,10)(-5,-5)
\put(0,0){\line(1,1){10}}
\put(0,0){\line(-1,1){10}}
\line(0,-1){13}
\put(-2.5,-2.5){$\bullet$}
\put(-2.5,-10){$\bullet$},
\end{picture},  \qquad \tau_{2} = \begin{picture}(10,10)(-5,-5)
\put(0,0){\line(1,1){10}}
\put(0,0){\line(-1,1){10}}
\line(0,-1){13}
\put(3,3){$\bullet$}
\put(-2.5,-2.5){$\bullet$}
\put(-2.5,-10){$\bullet$},
\end{picture}, \qquad \tau_{3} =  \begin{picture}(10,10)(-5,-5)
\put(0,0){\line(1,1){10}}
\put(0,0){\line(-1,1){12}}
\line(0,-1){13}
\put(-9,3.8){$\bullet$}
\put(-2.5,-2.5){$\bullet$}
\put(-2.6,-10){$\bullet$}
\end{picture}, \qquad \tau_{4} = \begin{picture}(10,10)(-5,-5)
\put(0,0){\line(1,1){15}}
\put(0,0){\line(-1,1){10}}
\line(0,-1){13}
 \put(2.5,2.5){$\bullet$}
 \put(7.5,7.5){$\bullet$}
\put(-2.5,-2.5){$\bullet$}
\put(-2.6,-10){$\bullet$},
\end{picture}, \qquad \tau_{5} =
\begin{picture}(10,10)(-5,-5)
\put(0,0){\line(1,1){12}}
\put(0,0){\line(-1,1){15}}
\line(0,-1){13}
\put(-9,3){$\bullet$}
\put(-8,4){\line(1,1){12}}
\put(-3,8.5){$\bullet$}
\put(3,3){$\bullet$}
\put(-2.7,-2.5){$\bullet$}
\put(-2.7,-10){$\bullet$}
\end{picture},
$$
$\tau_{1}$ is a \lft, while $\tau_{2}$, $\tau_{3}$, $\tau_{4}$ and $\tau_{5}$ are not.

For any $\tau \in \calt$, let $\Lf(\tau)$ denote the tree after replacing the right subtree of each bi-vertex of $\tau$ by a leaf. Then $\Lf(\tau)$ automatically satisfies condition~(\mref{it:lfta}) for a \lft.
By definition, for the trees $\tau_1,\cdots,\tau_5$ above, we have
$$
\Lf(\tau_{1}) =\tau_{1}, \quad \Lf(\tau_{2}) = \tau_{1},
\quad  \Lf(\tau_{3})= \tau_{3}, \quad \Lf(\tau_{4}) = \tau_{1},\quad \Lf(\tau_{5}) = \begin{picture}(10,10)(-5,-5)
\put(0,0){\line(1,1){12}}
\put(0,0){\line(-1,1){12}}
\line(0,-1){13}
\put(-8,4){\line(1,1){10}}
\put(-2.5,-10){$\bullet$}
\put(-2.5,-2.5){$\bullet$}
\put(-9,3){$\bullet$}
\end{picture}.
$$

Let $\tau = t_{P}^{s} \circ t_{\mu} \circ (\tau_{l} \ot \tau_{r})$ be a bracketed tree with $\lef_{\tau} \geq 2$. $\tau$ is called a {\bf \fat} if $\Lf(\tau)$ is a \lft and $\pb(\tau_{r}) \leq 1$. For example, $\Lf(\tau_{1}), \Lf(\tau_{2}), \Lf(\tau_{4})$ and $\Lf(\tau_{5})$ are \lfts, while $\Lf(\tau_{3})$ is not. It follows from $\pb((\tau_{4})_{r})=2$ that $\tau_{4}$ is not a \fat. Hence, $\tau_{1}, \tau_{2}$ and $\tau_{5}$ are \fats, while $\tau_{3}$ and $\tau_{4}$ are not.

The set $\avt$ of {\bf averaging trees} consists of all unreduced binary trees
satisfying the following conditions:
\begin{enumerate}
\item \label{it:a}Each bracketed subtree of $\tau$ is either a ladder or a \fat;
\item \label{it:b} For each bi-vertex $v$ of $\tau$, $\tau_{v,r}$ is either trivial or bracketed such that either $\tau_{v,l}$ is trivial or has a trivial right subtree.
\end{enumerate}

\begin{remark}
It follows from the definition of an averaging tree that a subtree of an averaging tree is still an averaging tree. The trivial tree and all ladders are averaging trees.
\mlabel{rk:sub}
\end{remark}

For example, consider the following eight unreduced binary trees

$$
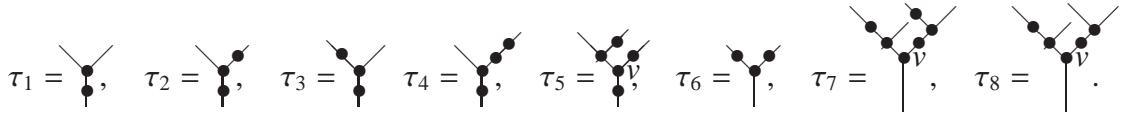
\tau_{1} = \begin{picture}(10,10)(-5,-5)
\put(0,0){\line(1,1){10}}
\put(0,0){\line(-1,1){10}}
\line(0,-1){13}
\put(-2.5,-2.5){$\bullet$}
\put(-2.5,-10){$\bullet$},
\end{picture},  \quad \tau_{2} = \begin{picture}(10,10)(-5,-5)
\put(0,0){\line(1,1){10}}
\put(0,0){\line(-1,1){10}}
\line(0,-1){13}
\put(3,3){$\bullet$}
\put(-2.5,-2.5){$\bullet$}
\put(-2.5,-10){$\bullet$},
\end{picture}, \quad \tau_{3} =  \begin{picture}(10,10)(-5,-5)
\put(0,0){\line(1,1){10}}
\put(0,0){\line(-1,1){12}}
\line(0,-1){13}
\put(-9,3.8){$\bullet$}
\put(-2.5,-2.5){$\bullet$}
\put(-2.6,-10){$\bullet$},
\end{picture} \quad \tau_{4} = \begin{picture}(10,10)(-5,-5)
\put(0,0){\line(1,1){15}}
\put(0,0){\line(-1,1){10}}
\line(0,-1){13}
 \put(2.5,2.5){$\bullet$}
 \put(7.5,7.5){$\bullet$}
\put(-2.5,-2.5){$\bullet$}
\put(-2.6,-10){$\bullet$},
\end{picture}, \quad \tau_{5} =
\begin{picture}(10,10)(-5,-5)
\put(0,0){\line(1,1){12}}
\put(0,0){\line(-1,1){15}}
\line(0,-1){13}
\put(-9,3){$\bullet$}
\put(-8,4){\line(1,1){12}}
\put(-3,8.5){$\bullet$}
\put(3,3){$\bullet$}
\put(-2.7,-2.5){$\bullet v$}
\put(-2.7,-10){$\bullet$}
\end{picture}, \quad \tau_{6} =  \begin{picture}(10,10)(-5,-5)
\put(0,0){\line(1,1){10}}
\put(0,0){\line(-1,1){10}}
\line(0,-1){13}
\put(-9.1,3){$\bullet$}
\put(-2.5,-2.5){$\bullet$}
\put(3,3){$\bullet$}
\end{picture},\quad \tau_{7} = \begin{picture}(20,20)(-10,-10)
\put(0,0){\line(1,1){20}}
\put(0,0){\line(-1,1){20}}
\line(0,-1){20}
\put(-2.5,-2.5){$\bullet v$}
\put(3,3){$\bullet$}
\put(8,8){$\bullet$}
\put(11,11){\line(-1,1){10}}
\put(2,14){$\bullet$}
\put(-9,3){$\bullet$}
\put(-8,4){\line(1,1){10}}
\put(-15,9){$\bullet$}
\end{picture},\quad \tau_{8} = \begin{picture}(20,20)(-10,-10)
\put(0,0){\line(1,1){20}}
\put(0,0){\line(-1,1){20}}
\line(0,-1){20}
\put(-2.5,-2.5){$\bullet v$}
\put(3,3){$\bullet$}
\put(8,8){$\bullet$}
\put(11,11){\line(-1,1){10}}
\put(-9,3){$\bullet$}
\put(-8,4){\line(1,1){10}}
\put(-15,9){$\bullet$}
\end{picture}.
$$
By the above rules, $\tau_{1}$ and $\tau_2$ are in $\avt$. $\tau_{3}$ and $\tau_{4}$ are bracketed subtrees of themselves respectively,  but they are either ladders nor \fats, so $\tau_3$ and $\tau_{4}$ are not in $\avt$. $\Lf(\tau_{5})$ is a \fat. ${\tau_{5}}_{v,r}$ is a ladder but the right subtree of ${\tau_{5}}_{v,l}$ is not trivial. Hence $\tau_5$ is not in $\avt$. Each bracketed subtree of $\tau_{6}$ is a ladder, but both ${\tau_{6}}_{v,r}$ and ${\tau_{6}}_{v,l}$ are bracketed. Then ${\tau_{6}}_{v,l}$ is not trivial or has a trivial right subtree. Thus $\tau_6$ is not in $\avt$. Since ${\tau_{7}}_{v,r} = \tau_{3}$ is not a \fat, the bracketed suctree ${\tau_{7}}_{v,r}$ of $\tau_{7}$ is not a \fat. Hence $\tau_7$ is not in $\avt$. There are two bracketed subtrees of $\tau_{8}$, one is a ladder, the other is equal to $\tau_{1}$. Since $\tau_{v,l}$ has a trivial right subtree, $\tau_{8}$ is in $\avt$.

There is a one to one correspondence between $AW(\{x\})$ and $\avt$. Recall that the number of $x$'s appearing in an averaging words $W$ is called its arity and denote it by $\ari_{W}$. We recursively define a map
$$
\varphi: AW(\{x\}) \longrightarrow \calt.
$$
We do this first by letting $\varphi(x) = \mathrm{id}$ and $\varphi(\lc x \rc^{s}) = t_{P}^{s}$. Assume that $\varphi$ has been defined for any word $W$ with $\ari_{W} \leq n$, $n \geq 1$. For any word $W$ with $\ari_{W}=n+1$, consider its standard decomposition $W = w_{1}w_{2}\cdots w_{m} \in AW$, where $w_{i} = x$ or  $\lc V_{i} \rc^{s_{i}}$ with $V_{i} \in \Omega^{+}$. When $m \geq 2$, define
\begin{eqnarray}\label{eq:correspond}
&\varphi(W) = t_{\mu} \circ (\varphi(w_{1}\cdots w_{m-1}) \ot \varphi(w_{m})).&
\end{eqnarray}
When $m=1$, $W$ is the form of $\lc V \rc^{s}$ with $V \in \Omega^{+}$, then $V$ is a word with $\ari_{V}=n+1$ and its breadth is $\geq 2$. Define
$$
\varphi(W) = t_{P}^{s} \circ \varphi(V).
$$

\begin{prop}
The map $\varphi$ is a bijection from $\AW(\{x\})$ to $\avt$.
\mlabel{pp:awat}
\end{prop}
\begin{proof}
We divide the proof into three steps. First, we show the image of $\varphi$ is a subset of $\avt$. Then we define a map $\Phi: \avt \longrightarrow \mathrm{AW}(\{x\})$ such that $\varphi \Phi = \mathrm{Id}_{\avt}$. Finally, we prove the map $\Phi$ also satisfies $\Phi \varphi = \mathrm{Id}_{AW(\{x\})}$.

\noindent
{\bf Step 1. $\varphi(\AW(\{x\})) \subseteq \avt$:} \quad  We prove it by induction on $\ari_{W}$. When $\ari_{W}=1$, that is $W= x$ or $\lc x \rc^{s}$. Then $\varphi(W)$ is a trivial tree or ladder. Hence, $\varphi(W) \subseteq \avt$.

Inductively assuming that $\varphi(W) \in \avt$ holds for any $W$ with $\ari_{W} \leq n$, $n \geq 1$. Let $W$ be any averaging word with $\ari_{W}=n+1$. Let $w_{1}\cdots w_{m}$ be the standard form of $W$, then we have $w_{i}$ is either $x$ or $\lc w_{i}'\rc$, $1\leq i \leq m$. According to the breadth of $W$, we consider the following two subcases:\quad (i) $m \geq 2$; \quad (ii) $m=1$.
\begin{enumerate}
\item[(i)] If $m \geq 2$, then we have $\tau:= \varphi (w_{1} \cdots w_{m}) = t_{\mu} \circ (\varphi (w_{1}\cdots w_{m-1}) \ot \varphi(w_{m}))$. Denote the root of $\tau$ by $v$. Then $\tau_{v,l} = \varphi (w_{1}\cdots w_{m-1})$, $\tau_{v,r} = \varphi(w_{m})$. By the induction hypothesis, we have $\tau_{v,l}, \tau_{v,r} \in \avt$. Since each bracketed subtree of $\tau$  is a bracketed subtree of either $\tau_{v,l}$ or $\tau_{v,r}$, $\tau$ satisfies (\ref{it:a}) for $\avt$.

    Except for the root $v$ of $\tau$, each bi-vertex of $\tau$ is a bi-vertex of either $\tau_{v,l}$ or $\tau_{v,r}$. In order to show $\tau$ satisfies (\ref{it:b}) for $\avt$, we only need to prove $v$ satisfies (\ref{it:b}) for $\avt$.
\begin{enumerate}
\item[($\alpha$)] If $w_{m}=x$, then $\tau_{v,r}$ is a trivial tree.
\item[($\beta$)] If $w_{m} = \lc w_{m}' \rc$, then $\varphi(w_{m})$ is a bracketed tree and $w_{m-1}=x$. Thus $\tau_{v,l}$ is trivial when $m=2$ and $\tau_{v,l}$ has a trivial right subtree when $m \geq 3$.
\end{enumerate}
We have $\tau$ satisfies (\ref{it:b}) for $\avt$. Therefore, $\tau \in \avt$.
\item[(ii)] If $m=1$, then $W = \lc V \rc^{s}$, where $V\in \Omega^{+}$. Note that $\ari_{W} \geq 2$ and by Remark \ref{rem:stand}. Rewrite $W$ as the form of  $\lc v_{1}\cdots v_{k}\rc^{s}$.  Then $k \geq 2$, $v_{1} =x$, and $v_{k} =x$ or $v_{k} = \lc v_{k}' \rc$, $v_{k}' \in \Omega^{+}$. Let $\tau:= t_{P}^{s} \circ \varphi(V) = \varphi(W)$.

    Except for $\sigma_{i}:=t_{P}^{i} \circ \varphi(V)$, $1 \leq i \leq s$, each bracketed subtree of $\tau$ is a bracketed subtree of $\varphi(V)$. We only need to prove $\sigma_{i}$ is a \fat. Since $w_{1} = x$ and $\pb({\sigma_{i}}_{r}) = \pb(\varphi(v_{k})) \leq 1$, $\sigma_{i}$ is a \fat and $\tau$ satisfies (\ref{it:a}) for $\avt$.

    Since $V=v_{1}\cdots v_{k}$ is an averaging word with $\ari_{V}=n+1$ and $k \geq 2$, (i) implies $\varphi(V) \in \avt$. It follows from each bi-vertex of $\tau$ is a bi-vertex of $\varphi(V)$ that $\tau$ satisfies (\ref{it:b}) for $\avt$. Therefore, $\tau \in \avt$.
\end{enumerate}

\noindent
{\bf Step 2. Define a map $\Phi: \avt \longrightarrow \AW(\{x\})$ such that $\varphi \Phi = \mathrm{Id}_{\avt}$:}

We do this by induction on $\lef_{\tau}$. When $\lef_{\tau}=1$, that is $\tau$ is either trivial or a ladder. Define
$$
\Phi(\mathrm{id}) =x \quad \mbox{and} \quad \Phi(t_{P}^{\ell} \circ \mathrm{id}) = \lfloor x \rfloor^{\ell}.
$$
We immediately have $\varphi \Phi(\id) = \varphi(x) =\id$ and $\varphi \Phi(t_{P}^{\ell} \circ \id) = \varphi(\lc x \rc^{\ell}) = t_{P}^{\ell} \circ \id$.

Assume that $\Phi$ has been defined for averaging trees $\tau$ with $\lef_{\tau} \leq$ $n$, $n \geq 1$ and $\varphi \Phi(\tau) = \tau$. For any averaging tree $\tau$ with $\lef_{\tau} = n+1$. Rewrite $\tau$ as $t_{P}^{s} \circ \tau'$ with $s\geq 0$ and $\tau'$ is non-bracketed. Then we have $\tau = t_{P}^{s} \circ \left(t_{\mu} \circ (\tau_{l}  \ot \tau_{r} ) \right)$ for some $\tau_{l}, \tau_{r}$, where $\tau_{r}$ is either trivial or bracketed. By Remark~\mref{rk:sub}, $\tau_{l}$ and $\tau_{r}$ are still averaging trees. By the induction hypothesis, there exist averaging words $W_{1}$ and $W_{2}$ such that $\Phi(\tau_{l}) = W_{1}$, $\varphi(W_{1}) = \tau_{l}$ and $\Phi(\tau_{r}) = W_{2}$, $\varphi(W_{2}) =\tau_{r}$ respectively. By the definition of $\varphi$ and the property of $\tau_{r}$, we have the breadth of $W_{2}$ is $1$. Define $\Phi(\tau) = \lc W_{1}W_{2}\rc^{s}$. Then we have
\begin{eqnarray*}
\varphi \Phi (\tau) &=& \varphi(\lc W_{1}W_{2}\rc^{s}) = t_{P}^{s} \circ \varphi(W_{1}W_{2})\\
&=& t_{P}^{s} \circ t_{\mu} \circ (\varphi(W_{1}) \ot \varphi(W_{2})) \quad \mbox{(the breadth of $W_{2}$ is $1$)}\\
&=& t_{P}^{s} \circ t_{\mu} \circ (\tau_{l} \ot \tau_{r}) = \tau.
\end{eqnarray*}
It remains to prove $\Phi$ is well-defined, that is we need to show $\lc W_{1} W_{2}\rc^{s} \in \AW(\{x\})$.

In fact, if $s=0$, by (\ref{it:b}) for $\avt$, we have $W_{2} =x$ or $W_{2} = \lc  V \rc$. If $W_{2} =x$, then $W=W_{1}x$ is still an averaging word.   If $W_{2} = \lfloor V \rfloor$, we have $\tau_{r} = t_{P}^{\ell} \circ \tau'_{r}$ is bracketed. First, $W_{1}$ is not the form $\lfloor V' \rfloor$, otherwise $\tau_{l}$ would be bracketed and then doesn't have right subtree. Second, if $W_{1} = w_{11}w_{12} \cdots w_{1t}$, then $w_{1t}$ is not the form $\lfloor V'' \rfloor$ (the tail of $W_{1}$), otherwise the right subtree of the root of $\tau_{l}$ would be bracketed (not trivial) and gives a contradiction. So we have $w_{1t} =x$ and $W=W_{1}W_{2}$ is still an averaging word.

If $s>0$, that is $\tau$ is bracketed. Let $W_{1} = w_{11}w_{12} \cdots w_{1t}$ be the standard decomposition of $W_1$. By (\ref{it:b}) for $\avt$, we have either $W_{2} =x$ or $W_{2} = \lc V \rc^{\ell}$, where $V\in \Omega^{+}$. If $W_{2} = \lc V \rc^{\ell}$ and $\ell \geq 2$, then $\tau_{r}$ is a bracketed tree with $\pb(\tau_{r}) = \ell \geq 2$. Thus $\tau$ is not a \fat. It gives a contradiction with $\tau$ is an averaging tree. So we have $\ell =1$.

Since $\tau$ is a bracketed tree, $\Lf(\tau)$ is a \fat and bracketed, we have $w_{11} =x$, otherwise there would be a bi-vertex of $\tau$ whose left subtree is bracketed. Hence $W_{1}W_{2}$ is an averaging word with the head $x$ and $W_{2}$ is either $x$ or a bracketed word of the form $\lc V \rc$, where $V \in \Omega^{+}$. That is $W_{1}W_{2} \in \Omega^{+}$. Therefore $\lfloor W_{1}W_{2}\rfloor^{s}$ is still an averaging word.

\noindent
{\bf Step 3. $\Phi$ satisfies $\Phi \varphi = \mathrm{Id}_{\AW(\{x\})}$:}

We prove it by induction on $\ari_{W}$. When $\ari_{W}=1$, that is $W =x$ or $\lc x \rc^{s}$. Then we have $\Phi\varphi(x) = \Phi(\mathrm{id}) =x$ and $\Phi \varphi(\lc x \rc^{s}) = \Phi(t_{P}^{s} \circ \mathrm{id}) = \lc x \rc^{s}$.

Assume that $\Phi \varphi = \mathrm{Id}_{\AW(\{x\})}$ holds for any $W$ with $\ari_{W} \leq n$, $n \geq 1$. Let $W$ be any averaging word with $\ari_{W}=n+1$. According to breadth of $W$, we consider the following two subcases:
\begin{enumerate}
\item[(1)] If $m \geq 2$ and $W= w_{1} \cdots w_{m}$ is in standard form. We have
\begin{eqnarray*}
\Phi \varphi (w_{1} \cdots w_{m}) &=& \Phi \big( t_{\mu} \circ (\varphi (w_{1}\cdots w_{m-1}) \ot \varphi(w_{m})) \big)\\
&=& \Phi (\varphi(w_{1}\cdots w_{m-1})) \Phi( \varphi(w_{m})) \quad \mbox{(by the definition of $\Phi$)}\\
&=& w_{1}\cdots w_{m-1} w_{m}. \quad \mbox{(by induction hypothesis)}
\end{eqnarray*}
\item[(2)] If $m=1$, then $W = \lc V \rc^{s}$, where $V\in \Omega^{+}$. Rewrite $W$ as the form of  $\lc v_{1}\cdots v_{k}\rc^{s}$.  Then $k \geq 2$, $v_{1} =x$, and $v_{k} =x$ or $v_{k} = \lc v_{k}' \rc$, $v_{k}' \in \Omega^{+}$. We have
\begin{eqnarray*}
\Phi \varphi (\lc v_{1} \cdots v_{k} \rc^{s}) &=& \Phi \big( t_{P}^{s} \circ \varphi (v_{1}\cdots v_{k-1}v_{k}) \big)\\
&=& \lc \Phi (\varphi(v_{1}\cdots v_{k-1}v_{k}) )\rc^{s} \quad \mbox{(by the definition of $\Phi$)}\\
&=& \lc v_{1} \cdots v_{k} \rc^{s}. \quad \mbox{(by case (1))}
\end{eqnarray*}
\end{enumerate}

Therefore, $\varphi$ is a bijection between $\AW(\{x\})$ and $\avt$.
\end{proof}

Now, we can give the free averaging associative algebra in terms of trees.
\begin{theorem}\mlabel{thm:freeone}
The free averaging associative algebra on one generator is $\oplus_{n \geq 1} \bfk[\avt_{n}]$, where $\avt_{n}$ is the set of averaging trees with $n$ leaves.
The binary operations on $\avt_{p} \times \avt_{q}$ and the averaging operator are given respectively by:
$$
\tau_{p} \ast \tau_{q} = \varphi (\varphi^{-1}(\tau_{p}) \diamond \varphi^{-1} (\tau_{q})), \quad Q(\tau) = \varphi ( P( \varphi^{-1}(\tau) ) ).
$$
\label{thm:avt}
\end{theorem}
It is cumbersome to describe the operations $\ast$ and $Q$ directly in terms of averaging trees. So we will not provide it here. On the other hand, both Proposition~\mref{pp:awat} and Theorem~\mref{thm:avt} can be easily generalized to free averaging algebras on any nonempty set $X$. Just replace $\avt$ by trees from $\avt$ with leaves decorated by elements of $X$.

\subsection{The operad of averaging associative algebras}
The operad of the averaging algebra (also called the averaging operad) is given as a quotient of the free operad in~\cite{PBGN}. We will recall its definition and apply the description of free averaging algebra on one generator in terms of unreduced binary trees in the previous subsection to give an explicit construction of the averaging operad.

\begin{defn}\mlabel{defn:av}
{\rm Let $\gensp=\gensp_{2}$ be a graded vector space concentrated at arity 2.
\begin{enumerate}
\item \cite[Section 5.8.5]{LV}
Let $\bvp$ denote the graded spaces concentrated at arity 1 and arity 2 with  $(\bvp)_{1}=\bfk\,\id\oplus \bfk\,P$ and $(\bvp)_{2}=V$, where $P$ is a symbol. Let $\mathcal{T}_{ns}(\bvp)$ be the free operad generated by $\bvp$.
\item \cite{PBGN}
Let $\calp=\mathcal{T}_{ns}(\gensp)/(R_\calp)$ be a binary operad defined by generating operations $\gensp$ and relations $R_\calp$.
Let
$$AV_{\mathcal{P}}:=\{\gop\circ(P\otimes P)-P\circ\gop\circ(P\otimes {\rm id}), \gop\circ(P\otimes P)- P\circ\gop\circ({\rm id}\otimes P)\ |\ \gop\in \gensp\}.$$
Define the {\bf operad of averaging $\calp$-algebras}
by
$$AV(\calp):=\mathcal{T}_{ns}(\bvp)/( {\mathrm R}_\calp,{\mathrm AV}_{\calp}),$$
where $( {\mathrm R}_\calp,{\mathrm AV}_{\calp})$ is the operadic ideal of $\mathcal{T}_{ns}(\bvp)$ generated by ${\mathrm R}_\calp\cup {\mathrm AV}_{\calp}$.
\item
Let $As$ denote the non-symmetric operad of associative algebras. We call $AV(As)$ the {\bf averaging operad}.
\end{enumerate}
}
\end{defn}

The operad $AV(As)$ is a non-symmetric operad and so is completely determined by the free nonunitary averaging associative algebra on one generator. So by Theorem~\mref{thm:avt}, we have

\begin{prop}
Let $\calp_{0} =0$ and $\calp_{n} = \bfk \avt_{n}$. Then
$\calp = (\calp_{0}, \calp_{1},\calp_{2}, \cdots)$ is the averaging operad.
\end{prop}

We end the paper by a description of the operad compositions of $\calp$ by making use of bijection $\varphi$ in Proposition~\mref{pp:awat} together with the following fact:
Let $I_{AV}$ be the operated ideal of the free operated algebra $\bfk\frakS(X)$ generated by elements of the form
$$ \lc u\rc \lc v\rc - \lc u\lc v\rc \rc,\quad  \lc \lc u\lc v \rc -\lc u\lc v\rc\rc, \quad u, v\in \bfk\frakS(X).$$
Then the quotient operated algebra $\bfk \frakS(X)/I_{AV}$ is the free averaging algebra on $X$. Thus we have $\bfk\frakS(X)/I_{AV}\cong \cala(X)$ as operated algebras. Let
$$\mathrm{Red}: \bfk\frakS(X)\to \bfk\frakS(X)/I_{AV}\cong \cala(X)$$
be the composition of the quotient map with the above isomorphism.

Now let $\tau\in \calp_m$, $\sigma\in \calp_n$ and $1\leq i\leq m$. To define the composition $\tau \circ_i \sigma$, first replace the $i$-th $x$ of $\varphi^{-1}(\tau)$ (from the left) by $\varphi^{-1}(\sigma)$ and denote the result by $W$.
Then we have $\tau \circ_{i} \sigma = \varphi(\mathrm{Red}(W))$.

\bigskip

\noindent {\bf Acknowledgements}: This work is supported by the National Natural Science Foundation of China (Grant No. 11371178) and the National Science Foundation of US (Grant No. DMS~1001855). The authors thank the anonymous referees for helpful comments.

\end{document}